\documentclass[dvipsnames, a4paper, 12pt]{article}
\usepackage[utf8]{inputenc}
\usepackage{placeins}

\usepackage{hyperref}
\usepackage[utf8]{inputenc}
\usepackage[T1]{fontenc}
\usepackage{amsmath}
\usepackage{amsthm}
\usepackage{amsfonts}
\usepackage{amssymb}
\usepackage{graphicx,caption,subcaption}
\usepackage{fancybox}
\usepackage{enumitem}
\usepackage{soul}
\usepackage{mathtools}
\usepackage{lmodern}
\usepackage{stmaryrd}
\usepackage{multirow}
\usepackage{multicol}
\usepackage{xcolor}
\usepackage{tabularx}
\usepackage{array}
\usepackage{colortbl}
\usepackage{varwidth}
\PassOptionsToPackage{dvipsnames,svgnames}{xcolor}     
\usepackage{wrapfig}
\usepackage[explicit,calcwidth]{titlesec}
\usepackage{fancyhdr}
\usepackage[left=2.1cm, right=2.1cm, top=2cm, bottom=2cm]{geometry}
\usepackage{esint}
\usepackage{tocloft}
\usepackage[skip=8pt plus1pt, indent=0pt]{parskip}
\usepackage{twoopt}
\usepackage{authblk}
\usepackage{accents}
\usepackage{mathrsfs}
\usepackage{adjustbox} %
\hypersetup{ colorlinks=true, %
    linktoc=all,     %
    linkcolor=blue,  %
} \PassOptionsToPackage{unicode}{hyperref}
\PassOptionsToPackage{naturalnames}{hyperref}
\usepackage{float}

\usepackage[ruled,vlined,linesnumbered]{algorithm2e}
\usepackage{todonotes}
\setuptodonotes{fancyline, color=red!20, shadow, inline}

\usepackage{silence}

\usepackage[bbgreekl]{mathbbol}
\DeclareSymbolFont{bbold}{U}{bbold}{m}{n}
\DeclareSymbolFontAlphabet{\mathbbold}{bbold}
\DeclareSymbolFontAlphabet{\mathbb}{AMSb}%

\usepackage[nameinlink, capitalise]{cleveref}

\newtheorem{theorem}{Theorem}[section]
\newtheorem*{theorem*}{Theorem}

\newtheorem{lemma}[theorem]{Lemma}
\newtheorem*{lemma*}{Lemma}

\newtheorem{remark}[theorem]{Remark}
\newtheorem{prop}[theorem]{Proposition}
\newtheorem{assumption}{Assumption}
\newtheorem{definition}[theorem]{Definition}

\usepackage[
    backend=biber,
    uniquename=false,
    bibencoding=utf8,
    sorting=nyt,
    doi=false, isbn=false, url=false,
    style=alphabetic,
    maxcitenames=2,
    uniquelist=false,
    maxbibnames=123,
    backref=true,
    hyperref=true
]{biblatex}
\usepackage{amsthm}
\usepackage[most]{tcolorbox}

\newtheoremstyle{upright}   %
  {10pt}                    %
  {10pt}                    %
  {}                        %
  {}                        %
  {\bfseries}               %
  {.}                       %
  { }                       %
  {}                        %

\theoremstyle{upright}

\tcbuselibrary{theorems}
\definecolor{theoremblue}{HTML}{2196F3} %
\definecolor{proofpurple}{HTML}{9013FE} %
\definecolor{defyellow}{HTML}{FFEB3B} %
\definecolor{assumptionred}{HTML}{ff3b3b} %
\definecolor{conventionred}{HTML}{ff3b3b} %
\definecolor{remarkorange}{HTML}{dc5800} %
\definecolor{examplegreen}{HTML}{00695c} %

\tcolorboxenvironment{theorem}{
  enhanced jigsaw,
  sharp corners,
  colframe=theoremblue,
  colback=theoremblue!5,
  coltitle=black,
  fonttitle=\bfseries,
  borderline west={2pt}{0pt}{theoremblue},
  before skip=10pt,
  after skip=10pt,
  boxrule=0pt,
  left=10pt,
  right=10pt,
  breakable
}

\tcolorboxenvironment{definition}{
  enhanced jigsaw,
  sharp corners,
  colframe=defyellow,
  colback=defyellow!5,
  coltitle=black,
  fonttitle=\bfseries,
  borderline west={2pt}{0pt}{defyellow},
  before skip=10pt,
  after skip=10pt,
  boxrule=0pt,
  left=10pt,
  right=10pt,
  breakable
}

\tcolorboxenvironment{corollary}{
  enhanced jigsaw,
  sharp corners,
  colframe=theoremblue,
  colback=theoremblue!5,
  coltitle=black,
  fonttitle=\bfseries,
  borderline west={2pt}{0pt}{theoremblue},
  before skip=10pt,
  after skip=10pt,
  boxrule=0pt,
  left=10pt,
  right=10pt,
  breakable
}

\tcolorboxenvironment{prop}{
  enhanced jigsaw,
  sharp corners,
  colframe=theoremblue,
  colback=theoremblue!5,
  coltitle=black,
  fonttitle=\bfseries,
  borderline west={2pt}{0pt}{theoremblue},
  before skip=10pt,
  after skip=10pt,
  boxrule=0pt,
  left=10pt,
  right=10pt,
  breakable
}

\tcolorboxenvironment{condition}{
  enhanced jigsaw,
  sharp corners,
  colframe=theoremblue,
  colback=theoremblue!5,
  coltitle=black,
  fonttitle=\bfseries,
  borderline west={2pt}{0pt}{theoremblue},
  before skip=10pt,
  after skip=10pt,
  boxrule=0pt,
  left=10pt,
  right=10pt,
  breakable
}

\tcolorboxenvironment{example}{
  enhanced jigsaw,
  sharp corners,
  colframe=examplegreen,
  colback=examplegreen!5,
  coltitle=black,
  fonttitle=\bfseries,
  borderline west={2pt}{0pt}{examplegreen},
  before skip=10pt,
  after skip=10pt,
  boxrule=0pt,
  left=10pt,
  right=10pt,
  breakable
}

\tcolorboxenvironment{lemma}{
  enhanced jigsaw,
  sharp corners,
  colframe=theoremblue,
  colback=theoremblue!5,
  coltitle=black,
  fonttitle=\bfseries,
  borderline west={2pt}{0pt}{theoremblue},
  before skip=10pt,
  after skip=10pt,
  boxrule=0pt,
  left=10pt,
  right=10pt,
  breakable
}

\tcolorboxenvironment{assumption}{
  enhanced jigsaw,
  sharp corners,
  colframe=assumptionred,
  colback=assumptionred!5,
  coltitle=black,
  fonttitle=\bfseries,
  borderline west={2pt}{0pt}{assumptionred},
  before skip=10pt,
  after skip=10pt,
  boxrule=0pt,
  left=10pt,
  right=10pt,
  breakable
}

\tcolorboxenvironment{remark}{
  enhanced jigsaw,
  sharp corners,
  colframe=remarkorange,
  colback=remarkorange!5,
  coltitle=black,
  fonttitle=\bfseries,
  borderline west={2pt}{0pt}{remarkorange},
  before skip=10pt,
  after skip=10pt,
  boxrule=0pt,
  left=10pt,
  right=10pt,
  breakable
}

\tcolorboxenvironment{convention}{
  enhanced jigsaw,
  sharp corners,
  colframe=conventionred,
  colback=conventionred!5,
  coltitle=black,
  fonttitle=\bfseries,
  borderline west={2pt}{0pt}{conventionred},
  before skip=10pt,
  after skip=10pt,
  boxrule=0pt,
  left=10pt,
  right=10pt,
  breakable
}

\newtcolorbox{myproof}{
  enhanced jigsaw,
  sharp corners,
  colframe=proofpurple,
  colback=proofpurple!0,
  borderline west={1pt}{0pt}{proofpurple},
  borderline south={1pt}{0pt}{proofpurple},
  before skip=10pt,
  after skip=10pt,
  boxrule=0pt,
  left=10pt,
  right=10pt,
  breakable,
}

\renewenvironment{proof}[1][\proofname]{
  \begin{myproof}
    \textbf{#1.}
}{\hfill$\qed$\end{myproof}
}
\newcommand{\N}[0]{\mathbb{N}}

\newcommand{\R}[0]{\mathbb{R}}
\newcommand{\C}[0]{\mathbb{C}}
\newcommand{\K}[0]{\mathcal{K}}

\newcommand{\V}[1]{\mathbb{V}\left[#1\right]}

\newcommand{\dd}[0]{\mathrm{d}}

\newcommand{\Tr}[0]{\mathrm{Tr}}

\DeclareMathOperator{\argmin}{argmin}

\DeclareMathOperator{\supp}{supp}
\newcommand{\W}[0]{\mathrm{W}}

\newcommand{\app}[4]{\left\lbrace\begin{array}{ccc}
   #1 & \longrightarrow & #2 \\
   #3 & \longmapsto & #4 \\
\end{array} \right.}

\newcommand{\oll}[1]{\overline{#1}}

\newcommand{\verteq}{\rotatebox{90}{$\,=$}}

\newcommand{\step}[2]{\textrm{---} \textit{Step #1}: #2}

\newcommand{\X}{\mathcal{X}}
\newcommand{\Y}{\mathcal{Y}}
\newcommand{\Z}{\mathcal{Z}}

\DeclareMathOperator{\Extr}{Extr}
\newcommand{\Leb}{\mathscr{L}}
\newcommand{\T}{\mathcal{T}}
\newcommand{\bures}{d_{\mathrm{BW}}^2}
\newcommand{\PD}[1][d]{S_{#1}^{++}(\R)}

\DeclareMathOperator{\KL}{KL}
\DeclareMathOperator{\NWC}{NWC}

\definecolor{mybluei}{HTML}{ABE6FF}

\newcommand{\blue}[1]{{\color{BlueViolet}#1}}

\title{Computing Barycentres of Measures for Generic Transport Costs}
\author[1]{Eloi Tanguy}
\author[1]{Julie Delon}
\author[1]{Nathaël Gozlan}
\affil[1]{Universit\'e Paris Cit\'e, CNRS, MAP5, F-75006 Paris, France}
\date{16th July 2026}

\renewcommand{\blue}[1]{#1}

\begin{document}
\maketitle

\begin{abstract}
    Wasserstein barycentres represent average distributions between multiple probability measures for the Wasserstein distance. 
The numerical computation of Wasserstein barycentres is notoriously challenging. A common approach is to use Sinkhorn iterations, where an entropic regularisation term is introduced to make the problem more manageable. Another approach involves using fixed-point methods, akin to those employed for computing Fréchet means on manifolds. The convergence of such methods for 2-Wasserstein barycentres, specifically with a quadratic cost function and absolutely continuous measures, was studied by Alvarez-Esteban et al. in~\cite{alvarez2016fixed}. In this paper, we delve into the main ideas behind this fixed-point method and explore how it can be generalised to accommodate more diverse transport costs and generic probability measures, thereby extending its applicability to a broader range of problems. We show convergence results for this approach and illustrate its numerical behaviour on several barycentre problems.
\end{abstract}

\tableofcontents

\section{Introduction}
\label{sec:intro}

\subsection{Related Works and Motivation}

Wasserstein barycentres represent a powerful concept in Optimal Transport
theory, enabling the computation of average distributions between multiple
probability measures. These barycentres preserve the geometric structure of the
underlying distributions, making them particularly suited for machine learning
tasks. They have proven useful in numerous applications, including image
processing~\cite{rabin2012wasserstein}, computer
graphics~\cite{solomon2015convolutional,bonneel2016wasserstein},
statistics~\cite{bigot2019penalization},  domain
adaptation~\cite{montesuma2021wasserstein}, generative
modelling~\cite{korotin2022wasserstein}, fairness in machine
learning~\cite{gordaliza2019obtaining} or model selection in Bayesian
learning~\cite{Backhoff-Veraguas:2018aa}. Wasserstein barycentres are also at
the core of clustering methods such as K-means, to define centroids in spaces of
probability measures~\cite{ho2017multilevel,mi2018variational}.

The classical notion of barycentre refers to the weighted average of a set of
points $(x_k)$ with positive weights $(\lambda_k)$ summing to $1$, in a metric
space $(E,d)$. Formally, a barycentre $\bar{x}$ is a point that minimises the
weighted sum of (typically squared) distances:
$$\bar{x} \in \underset{x\in E}{\argmin} \sum_{k=1}^K \lambda_k d^2(x, x_k).$$
This concept can be extended to the space of probability measures, where $d$ can
be replaced for instance by a transportation cost $\mathcal{T}_c$. We remind
that for two probability measures $\mu$ and $\nu$ on metric spaces $(\X,d_\X)$
and $(\Y,d_\Y)$, and a cost function $c:\X\times \Y \rightarrow \R_+$, the
optimal transport cost between $\mu$ and $\nu$ for the ground cost $c$ is
defined as
$$\T_c(\mu, \nu) = \underset{\pi\in\Pi(\mu, \nu)}{\inf}\int_{\X\times \Y} c \dd
\pi,$$ where $\Pi(\mu, \nu)$ is the set of probability measures on $\X\times
\Y $ with marginals $\mu$ and $\nu$. Considering $K$ different cost
functions $c_k$, the barycentre problem can be written in this setting as
\begin{equation}
    \bar{\mu} \in \underset{\mu}{\argmin} \sum_{k=1}^K \lambda_k \T_{c_k}(\mu, \nu_k). \label{eq:bary_intro}
\end{equation}
 When $(\X,d_\X) = (\Y,d_\Y)$ is a Polish space and  $c = d_\X^p$ with $p \ge
1$, $\W_p(\mu,\nu):=\left(\T_{d^p}(\mu, \nu)\right)^{\frac 1 p}$ defines a
distance between probability measures (with finite moment of order $p$), called
$p$-Wasserstein distance. In this case, the barycentre $\bar{\mu}$ defined
above is called a Wasserstein barycentre. Generalisation to a barycentre of a
probability measure on $\mathcal{P}(\X)$ and the consistency of their discrete
approximations is also studied by several authors~\cite{agueh2017vers}.

The theoretical analysis of Wasserstein barycentres begins with the foundational
work by Carlier and Ekeland~\cite{carlier2010matching}, who studied the
existence, uniqueness and dual formulations for barycentre problems with generic
continuous cost functions. Subsequent work by~\cite{agueh2011barycenter}
re-established the existence and dual formulations  of such barycentres for the
quadratic Wasserstein distance $\W_2$ on Euclidean spaces, and showed uniqueness
under the hypothesis that one of the original measures is absolutely continuous.
More recent studies have broadened these results:
~\cite{Carlier2023Wassersteinmedians} extended the theoretical analysis to
Wasserstein medians ($\W_1$), studying their stability properties, and
investigated dual and multi-marginal formulations. \cite{brizzi2024p} further
extended the framework to $\W_p$ distances for $p > 1$, proving existence and
uniqueness of barycentres for absolutely continuous measures on $\mathbb{R}^d$.
A follow-up study by~\cite{brizzi2024hwassersteinbarycenters} analysed the
general case for strictly convex and $\mathcal{C}^2$ cost functions with
non-degenerate Hessian.

From a computational perspective, calculating Wasserstein barycentres is known
to be a highly challenging problem, classified as NP-hard. According
to~\cite{altschuler2021wasserstein}, although polynomial-time algorithms exist
for computing Wasserstein barycentres with a fixed number of points, their
computational complexity scales exponentially with respect to the dimension of
the space, or with respect to the number of marginals. This makes direct
computation infeasible for high-dimensional problems or large sets of
distributions, which are common in practical applications.

To tackle these computational challenges, several approximate methods have been
developed for Wasserstein barycentres. The first paper to propose an algorithmic
solution for computing these barycentres was~\cite{rabin2012wasserstein}, which
computed Sliced Wasserstein barycentres through a gradient descent approach.
This method leveraged the sliced Wasserstein distance to achieve an efficient
approximation, significantly simplifying  computations.

A natural approach to develop easily computable approximations of such
barycentres is to replace  transport costs $\T_c$ by regularised versions 
$$\T_{c,\varepsilon}(\mu, \nu) = \underset{\pi\in\Pi(\mu, \nu)}{\inf}
\int_{\X\times \Y} c \dd \pi +\varepsilon \KL(\pi|\mu \otimes \nu),$$ as
proposed in~\cite{cuturi14fast}. When the support of the distributions and
barycentre is fixed (a grid for instance), the problem can be rewritten as a KL
projection problem and the so-called entropic barycentre can be computed
efficiently with a modified version of Sinkhorn's
algorithm~\cite{benamou2015iterative,computational_ot}. 

In order to deal with distributions without imposed support a second approach
also described in~\cite{cuturi14fast} relies on a fixed-point algorithm inspired
by the computation of Fréchet means on manifolds. Each step of this fixed point
approach consists in replacing the current barycentre $\mu$ by its image measure
by the map $\sum_{k=1}^K \lambda_k T_k$, where the $T_k$ are optimal maps
between $\mu$ and $\nu_k$ (assuming these maps exist). The authors
of~\cite{alvarez2016fixed} were the first to establish a rigorous proof of
convergence for this fixed-point approach in the case of absolutely continuous
measures $\nu_k$: more precisely, they proved convergence of a subsequence to a
fixed point and showed that if the fixed point is unique, it is indeed a
barycentre. Their study focuses specifically on the case of $\W_2$ barycentres,
with applications demonstrated mainly on Gaussian measures. Although their proof
is only provided for absolutely continuous measures, this fixed point approach
is frequently used for discrete measures and probably the baseline free-support
method provided in numerical optimal transport libraries~\cite{flamary2021pot}.
Building on the same ideas as~\cite{alvarez2016fixed}, the author
of~\cite{lindheim2023simple} extends the investigation of the fixed point
algorithm for discrete measures on $\R^d$,  limited to just one single
iteration, and deriving a worst-case error bound in the $\W_2$ and $\W_1$
settings. The iterative solver of ~\cite{alvarez2016fixed} has also been
extended in high dimensional settings by~\cite{korotin2022wasserstein}, which
use a neural solver for computing the optimal maps $T_k$. 

In closely related directions, several other approaches have been proposed to
compute Wasserstein barycentres over Riemannian manifolds~\cite{Kim-Pass17}, or
Gromov-Wasserstein barycentres~\cite{beier2025tangential,beier2023multi} and the
approach we develop in this paper share similarities
with~\cite{beier2025tangential}.

\subsection{Contributions and Outline}

In this paper, we develop a fixed-point approach to compute barycentres between
probability measures for generic transport costs, i.e. solutions of the
optimisation problem~\eqref{eq:bary_intro}. Our only hypotheses are that we work
on compact spaces, and that the ground costs $c_k$ are continuous and such that
$\argmin_x \sum_{k=1}^K \lambda_k c_k(x,x_k)$ is uniquely defined. In
particular,  we do not assume existence of optimal  transport maps between $\mu$
and the $\nu_k$, and we do not assume anything on the probability measures $\mu$
and $\nu_k$. We propose an iterative fixed-point algorithm
generalising~\cite{alvarez2016fixed} in this generic case. We show that the
sequences generated by this algorithm have converging sub-sequences, that limits
must be fixed-points of a certain mapping $G$, and that a barycentre
for~\eqref{eq:bary_intro} is also a fixed point of $G$. 

Numerically, we show that our approach specifically allows to extend the recent
definition of generalised Wasserstein barycentres presented
in~\cite{delon2021generalized}, notably by considering non-linear functions
between the ambient space and the subspaces of measures $\nu_k$. It also enables
efficient computation of barycentres for the mixture Wasserstein
metric~\cite{delon2020wasserstein}, which until now were calculated using their
multi-marginal equivalent formulation. \blue{Our methods provide new techniques
for computing barycentres that are robust to outlier measures (as investigated
on simple examples in \cref{sec:ex_bar_pq,sec:ex_colour_transfer}), which
provides alternatives to the proposals from
\cite{Carlier2023Wassersteinmedians,bartl2025robust}.}

The paper is organised as follows. In \cref{sec:bar_props}, we introduce a novel
notion of Optimal Transport barycentres in a certain space between measures
$\nu_k$ on potentially different spaces for generic costs $c_k$. In
\cref{sec:fp_algo}, we propose a fixed-point algorithm which generalises
\cite{alvarez2016fixed} and converges to solutions (in a certain sense). We
re-write the problem in a discrete setting in \cref{sec:discrete} and illustrate
our method in \cref{sec:numerics} on several numerical examples,
\href{https://github.com/eloitanguy/ot\_bar}{providing a publicly available
Python toolkit}.

\section{Lifting Ground Barycentres to Measures}\label{sec:bar_props}

We work with probability measures $\nu_k$ on compact metric spaces $(\Y_k,
d_{\Y_k})_{k\in \llbracket 1, K \rrbracket}$,\footnote{where for $a,b\in \N,\;
\llbracket a, b \rrbracket$ denotes the set $\{a, a+1, \cdots, b\}.$} of which
we will seek a ``barycentre'' $\mu$ in a compact metric space $(\X, d_\X)$. To
compare a measure $\nu_k \in \mathcal{P}(\Y_k)$ and $\mu \in \mathcal{P}(\X)$ we
consider continuous cost functions $c_k: \X\times\Y_k \longrightarrow \R_+$. A
barycentre will be a minimiser of the sum of the transport costs with respect to
the measure $\nu_k$, leading to the following energy for a measure $\mu \in
\mathcal{P}(\X)$:
\begin{equation}\label{eqn:def_V}
    V(\mu) := \sum_{k=1}^K\T_{c_k}(\mu, \nu_k),
\end{equation}
hence our minimisation problem reads
\begin{equation}\label{eqn:def_bar}
    \underset{\mu\in \mathcal{P}(\X)}{\argmin}\ V(\mu).
\end{equation}
Note that to introduce barycentre weights $\lambda_k$, it suffices to replace
$c_k$ with $\lambda_k c_k$, which allows us to include weights in the costs and
alleviate notation. We summarise our standing assumptions on the spaces and
costs in \cref{ass:spaces_costs}:

\begin{assumption}\label{ass:spaces_costs}
    The metric spaces $(\X, d_\X)$ and $(\Y_k, d_{\Y_k})$ are compact, and the
    costs $c_k: \X\times\Y_k \longrightarrow \R_+$ are continuous.
\end{assumption}

Existence of solutions for Problem \eqref{eqn:def_bar} was established by
\cite[Proposition 2]{carlier2010matching} under \cref{ass:spaces_costs}.

\begin{remark}\label{rem:uniqueness_matching_teams} Uniqueness was proven in
    \cite[Proposition 4]{carlier2010matching} if, essentially, for at least one
    $k$, the problem $\T_{c_k}(\mu, \nu_k)$ has a Monge solution, for which they
    assume that each $\nu_k$ is absolutely continuous on  $\Y_k = \oll{\Omega}$
    with $\Omega$ an open and bounded subset of $\R^d$ with with $\nu_k(\partial
    \Omega) = 0$. They also assume that the costs $c_k(\cdot, y)$ are Lipschitz
    with a uniform constant $L$ and that $c_k$ verifies the Twist condition:
    $c_k(\cdot, y)$ is differentiable, with $\partial_xc_k(x, \cdot)$ injective.
\end{remark}

The definition of a barycentre between measures $\nu_k$ can be seen as a lifting
of a notion of barycentre within $\X$ of points $(y_1, \cdots, y_K) \in \Y_1
\times \cdots \times \Y_K$. To give mathematical meaning to this intuition and
to our method, we will make the following assumption throughout the paper:
\begin{assumption}\label{ass:B}
    For all $(y_1, \cdots, y_K)\in \Y_1 \times \cdots \times \Y_K,$ the set
    $\underset{x\in \X}{\argmin}\ \sum_{k=1}^Kc_k(x, y_k)$ has a unique element.
\end{assumption}
The uniqueness of the optimisation problem in \cref{ass:B} allows us to
introduce the ground barycentre function $B$:
\begin{equation}\label{eqn:def_B}
    B: \app{\Y_1 \times \cdots \times \Y_K}{\X}{(y_1, \cdots, y_K)}{\underset{x\in \X}{\argmin}\ \sum_{k=1}^Kc_k(x, y_k).}
\end{equation}
For example, in the case $\X = \Y_1 = \cdots = \Y_K = \R^d$, with $c_k(x,
y_k) = \|x - y_k\|_2^2$, the ground barycentre function $B$ is the standard
Euclidean barycentre: $B(y_1, \cdots, y_K) = \frac{1}{K}\sum_{k=1}^Ky_k$. For
convenience, we introduce $\Y := \Pi_k\Y_k,$ equipped with the product distance,
with the notation $Y := (y_1, \cdots, y_K)$ for an element of $\Y$, as well as
the total cost function:
\begin{equation}\label{eqn:def_C}
    C:= \app{\X \times \Y}{\R_+}{(x, y_1, \cdots, y_K)}{\sum_{k=1}^Kc_k(x, y_k)}.
\end{equation}
Equipped with these convenient notations, we can write the multi-marginal
formulation of our barycentre problem:
\begin{equation}\label{eqn:multi_marginal}
    \underset{\pi \in \Pi(\nu_1, \cdots, \nu_K)}{\argmin} \int_{\Y}C(B(Y), Y)\dd\pi(Y).
\end{equation}
The barycentre problem defined in \cref{eqn:def_bar} is related to the
multi-marginal formulation through the following equation, due to
\cite[Proposition 3.3]{carlier2010matching}:
\begin{equation}\label{eqn:multi_marginal_bar_equivalence}
    \underset{\mu\in \mathcal{P}(\X)}{\argmin}\ V(\mu) = B \# \underset{\pi \in \Pi(\nu_1, \cdots, \nu_K)}{\argmin} \int_{\Y}C(B(Y), Y)\dd\pi(Y),
\end{equation}
where $\#$ denotes the push-forward operator: $f\#\mu :=
\mathrm{Law}_{X\sim\mu}[f(X)]$. The following technical result uses the
continuity of the $c_k$ and \cref{ass:B} to show that $B$ is continuous.

\begin{lemma}\label{lemma:B_C0} The function $B: \Y \longrightarrow \X$ defined
    in \cref{eqn:def_B} is continuous.
\end{lemma}
\begin{proof}
    The proof uses standard compactness arguments, showing that  for $Y_n
    \xrightarrow[n\longrightarrow+\infty]{}Y \in \Y$, $(B(Y_n))$ can only have $B(Y)$ as a subsequential limit.
\end{proof}

Another important technical result is the regularity of transport costs, which
we will use repeatedly. We gather well-known results in
\cref{lemma:Tc_regularity}.
\begin{lemma}\label{lemma:Tc_regularity} Consider $E, F$ compact metric spaces
    and let $c: E\times F \longrightarrow \R_+$ a measurable cost function. The
    optimal transport cost $\T_c$ has the following regularity for the weak
    convergence of measures depending on $c$:
    \begin{enumerate}
        \item If $c$ is lower-semi-continuous, then $\T_c$ is
        lower-semi-continuous.
        \item If $c$ is continuous, then $\T_c$ is continuous.
        \item If $E=F$ and $c$ is l.s.c. with $c(x,y) =0
        \Longrightarrow x= y$, then $\T_c(\mu, \nu) = 0 \Longrightarrow \mu =
        \nu$.
    \end{enumerate}
\end{lemma}
\begin{proof}
    Regarding item 1), by \cite[Theorem 1.42]{santambrogio2015optimal},
    Kantorovich duality holds for $c$ l.s.c. and thus $\T_c$ can be written as a
    supremum of l.s.c. functions, hence is l.s.c.. For item 2), the result is
    verbatim \cite[Theorem 1.51]{santambrogio2015optimal}. For item 3), if
    $\T_c(\mu, \nu) = 0$ then there exists $\pi \in \Pi(\mu, \nu)$ such that
    $\int_{E^2}c(x,y) \dd \pi(x, y) = 0$ (existence follows from lower
    semi-continuity, as in \cite[Theorem 1.5]{santambrogio2015optimal}).  Thus
    for $\pi$-almost-every $(x, y)$, $c(x, y) = 0$, which by assumption gives
    $x=y$, hence (using the same technique as in \cite[Proposition
    5.1]{santambrogio2015optimal}) for any test function $\phi \in
    \mathcal{C}^0(E, \R)$:
    $$\int_E \phi(x)\dd\mu(x) = \int_{E^2}\phi(x)\dd\pi(x,y) =
    \int_{E^2}\phi(y)\dd\pi(x,y)= \int_E\phi(y)\dd\nu(y), $$
    which shows that $\mu = \nu$.
\end{proof}

\begin{remark}
    Throughout this work, we work in the compact setting (see
    \cref{ass:spaces_costs}), which alleviates substantial technicalities. For
    the costs $c(x, y) := \|x-y\|^p$ for some norm $\|\cdot\|$ on $\R^d$, we
    believe that our results can be shown with careful assumptions on the
    moments of the measures (and convergence of the sequence of moments in
    addition to weak convergence).
\end{remark}

\section{A Fixed-Point Algorithm}\label{sec:fp_algo}

\subsection{Algorithm Definition}

In this section, we define a sequence $(\mu_t) \in \mathcal{P}(\X)^\N$ that will
approach a barycentre of fixed measures $\nu_k \in \mathcal{P}(\Y_k)$. We
propose a modified version of the iterated scheme from \cite{alvarez2016fixed}
to solve \cref{eqn:def_bar}. To define an iteration mapping, for $\mu\in
\mathcal{P}(\X)$, we consider the set of multi-marginal couplings
\begin{equation}\label{eqn:def_Gamma_mu}
    \Gamma(\mu) := \left\{\gamma \in \mathcal{P}(\X \times \Y_1 \times \cdots \times \Y_K)\ :\ \forall k \in \llbracket 1, K \rrbracket,\; \gamma_{0, k} \in \Pi_{c_k}^*(\mu, \nu_k)\right\},
\end{equation}
where, for all $k$, $\gamma_{0,k}$ denotes the $\X \times \Y_k$ marginal of
$\gamma$ and $\Pi_{c_k}^*(\mu, \nu_k)$ denotes the set of all optimal couplings
for the transport problem between $\mu$ and $\nu_k$ associated to the cost
function $c_k$. The existence of such multi-couplings is a consequence of the
well-known ``gluing lemma'' (see \cite[Lemma 5.5]{santambrogio2015optimal}). The
following multi-coupling provides an explicit element of $\Gamma(\mu)$ given
$\pi_k \in \Pi_{c_k}^*(\mu, \nu_k)$:
\begin{equation}\label{eqn:glue}
    \gamma(\dd x, \dd y_1, \cdots, \dd y_K) := \mu(\dd x)\pi_1^x(\dd y_1) \cdots \pi_K^x(\dd y_K),
\end{equation}
where we wrote the disintegration of $\pi_k$ with respect to its first marginal
$\mu$ as $\pi_k(\dd x, \dd y_k) = \mu(\dd x) \pi_k^x(\dd y_k)$. By abuse of
notation, we will denote $B\#\gamma := B\#\gamma_{1, \cdots, K}$, where
$\gamma_{1, \cdots, K} \in \mathcal{P}(\Y_1 \times \cdots \times \Y_k)$ is the
marginal of $\gamma$ with respect to $(y_1, \cdots, y_K)$. In terms of random
variables, if $(X, Y_1, \cdots, Y_K) \sim \gamma$, then $B\#\gamma =
\mathrm{Law}[B(Y_1, \cdots, Y_K)]$. Denoting $B\#\Gamma(\mu) := \{B\#\gamma,\;
\gamma \in \Gamma(\mu)\}$, we define the multi-valued mapping $G$ which maps
$\mu\in \mathcal{P}(\X)$ to the set of next iterates $G(\mu) \subset
\mathcal{P}(\X)$:
\begin{equation}\label{eqn:def_G}
    G := \left\lbrace\begin{array}{ccc}
    \mathcal{P}(\X) & \rightrightarrows & \mathcal{P}(\X) \\
    \mu & \mapsto & B\#\Gamma(\mu) \\
    \end{array} \right. .
\end{equation}
Note that this construction is similar to that of \cite[Remark
3.4]{alvarez2016fixed}. 
Moreover, the candidate barycentre $\oll{\mu} = B\#\gamma_{1, \cdots, K}$ is
closely related to the multi-marginal formulation of the barycentre problem (see
\cref{eqn:multi_marginal_bar_equivalence}). Indeed, set $\pi := \gamma_{1,
\cdots, K} \in \Pi(\mu_1, \cdots, \mu_K)$, notice that $\pi$ is a candidate for
the multi-marginal problem of a particular structure induced by the reference
measure $\mu$. In the case where the plans $\gamma_{0, k}$ are induced by maps
$T_k$, then this structure is the coupling $(T_1, \cdots, T_K)\#\mu$. In terms
of random variables, if $X\sim\mu$, then the chosen coupling is $(T_1(X),
\cdots, T_K(X))$.

Taking inspiration from the $\W_2^2$ case, we can see informally the iterate
$\oll{\mu} \in G(\mu)$ as a local linearisation of $\mathcal{P}(\X)$. To
illustrate this intuition, we consider the case $\X = \Y_1 = \cdots = \Y_K$ and
assume that for each $k$, the set of optimal plans $\Pi_{c_k}^*(\mu, \nu_k)$ is
reduced to $(I, T_k)$, or in other words, that the Monge problem has a unique
solution. Informally, one may see the set of maps $T: \X \longrightarrow \X$
sending $\mu$ to a measure $T\#\mu \in \mathcal{P}(\X)$ as the tangent space to
$\mathcal{P}(\X)$ at $\mu$. As a result, the problem of finding a barycentre
$\oll{\mu}$ can be seen from the viewpoint of the reference measure $\mu$ in the
tangent space $T_\mu\mathcal{P}(\X)$ as the problem of finding $S \in
T_\mu\mathcal{P}(\X)$ such that $S\#\mu$ would minimise the cost $V$. Our
approach takes a barycentre of the optimal maps $T_k$ by choosing the candidate
$S := B \circ (T_1, \cdots, T_K)$. In the case of the squared-Euclidean cost on
the common space $\R^d$, this amounts to $S := \sum_k \lambda_k T_k$, which is
exactly the Linearised Optimal Transport barycentre approximation for the
reference measure $\mu$, as introduced in \cite[Section
4.3]{merigot2020quantitative}. We illustrate this viewpoint schematically in
\cref{fig:fixed_point_linearisation}.
\begin{figure}[ht]
    \centering
    \includegraphics[width = 0.6\linewidth]{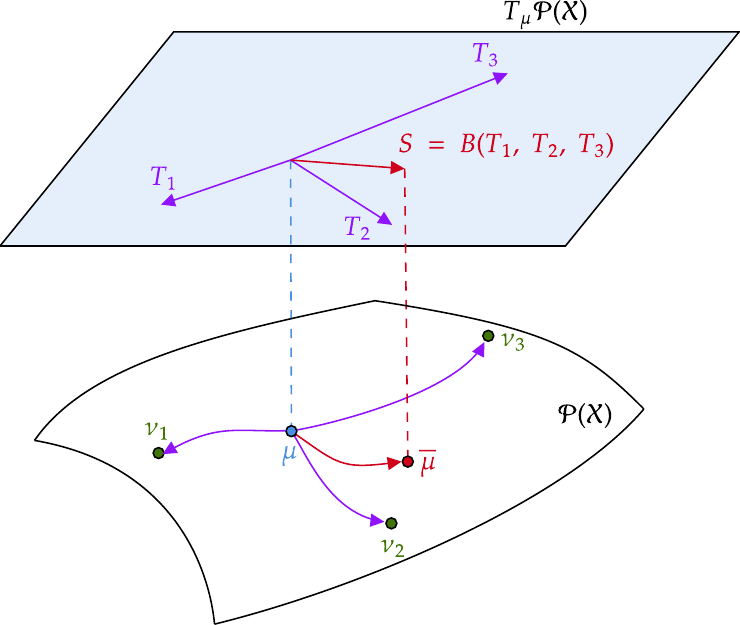}
    \caption{Illustration of the informal linearisation interpretation for the
    barycentre candidate $\oll{\mu} = B \circ (T_1, \cdots, T_K)\#\mu$.}
    \label{fig:fixed_point_linearisation}
\end{figure}

Starting from a measure $\mu_0 \in \mathcal{P}(\X)$, our algorithm consists of
choosing iterates through the multi-function $G$:
$$\forall t \in \N,\; \mu_{t+1} \in G(\mu_t).$$ We dedicate the next section to
a theoretical study of the convergence of this fixed-point iteration.

\subsection{Convergence of Fixed-Point Iterations}\label{sec:fp}

We can formulate a regularity result of the multi-valued map $G$: namely, we
will show that $G$ is \textit{upper hemi-continuous}. For the sake of
simplicity, we will take the following definition\footnote{We refer to
\cite[Chapter 17]{Aliprantis1994} for a more general definition and introduction
to these concepts on Polish spaces. We choose a stronger sequential definition
from \cite[Theorem 17.20]{Aliprantis1994}, which in their vocabulary corresponds
to u.h.c multi-functions with compact values.}:

\begin{definition}\label{def:uhc} A multi-valued function $\varphi: E
    \rightrightarrows F$ from a compact metric space $E$ to parts of a compact
    metric space $F$ is said to be \textit{upper hemi-continuous} (u.h.c.) if
    for any sequence $(x_n, y_n) \in (E\times F)^\N$ such that $y_n \in
    \varphi(x_n)$ and $x_n \xrightarrow[n\longrightarrow+\infty]{}x\in E$, there
    exists an extraction such that
    $y_{\alpha(n)}\xrightarrow[n\longrightarrow+\infty]{} y\in F$ with $y\in
    \varphi(x)$. 
\end{definition}

For more technical reasons, we also need to introduce the notion of
\textit{lower hemi-continuity}\footnote{whose formulation is is equivalent to
\cite[Definition 17.2]{Aliprantis1994}, by \cite[Theorem
17.21]{Aliprantis1994}.}

\begin{definition}\label{def:lhc} A multi-valued function $\varphi: E
    \rightrightarrows F$ from a compact metric space space $E$ to parts of a
    compact metric space space $F$ is said to be \textit{lower hemi-continuous}
    (l.h.c.) if for any sequence $(x_n)\in E^\N$ such that $x_n
    \xrightarrow[n\longrightarrow+\infty]{}x\in E$, then for any $y \in F$ such
    that $y \in \varphi(x)$, there exists an extraction $\alpha$ and a sequence
    $(y_n) \in F^\N$ such that $y_n \in \varphi(x_{\alpha(n)})$ and
    $y_{n}\xrightarrow[n\longrightarrow+\infty]{} y$.  
\end{definition}

To illustrate the technical differences between these two notions, we consider
two specific multi-valued functions in \cref{fig:ex_uhc_lhc}.
\begin{figure}[ht]
    \centering
    \includegraphics[width=.6\linewidth]{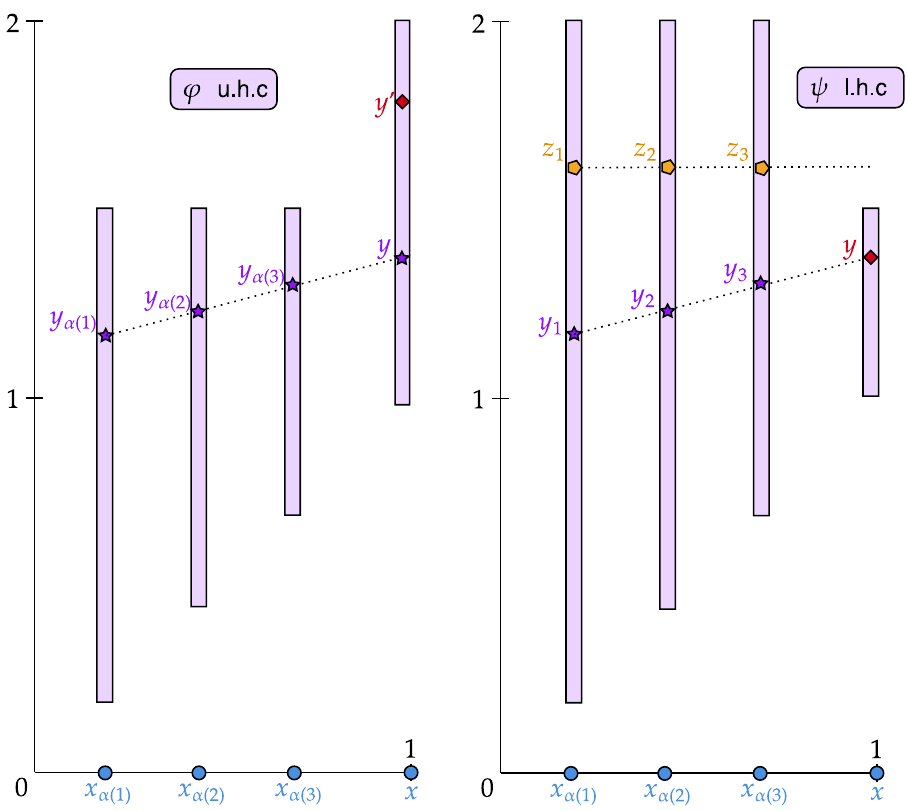}
    \caption{\textit{Left:} the multi-function $\varphi : [0, 1]
    \rightrightarrows [0, 2]$ defined by $\forall x \in [0, 1),\; \varphi(x) =
    [x, 3/2]$ and $\varphi(1) = [1, 2]$ is u.h.c.. Indeed, taking any sequence
    $(x_n, y_n)$ such that $y_n \in \varphi(x_n)$ and $x_n
    \xrightarrow[n\longrightarrow +\infty]{} x$, there exists an extraction
    $\alpha$ such that $y_{\alpha(n)} \xrightarrow[n \longrightarrow +\infty]{}
    y \in \varphi (x)$. However, $\varphi$ is not l.h.c. at 1 since the target
    $y' := 7/4 \in \varphi(1)$ can never be a limit of a sequence $(x_n, y_n)$
    with $x_n \xrightarrow[n\longrightarrow +\infty]{} 1$ and
    $y_n\in\varphi(x_n)$. \\\hspace{\textwidth} \textit{Right:} $\psi : [0, 1]
    \rightrightarrows [0, 2]$ defined by $\forall x \in [0, 1),\; \psi(x) = [x,
    2]$ and $\psi(1) = [1, 3/2]$ is l.h.c.. Take $x_n
    \xrightarrow[n\longrightarrow +\infty]{} x$ and a target $y \in \psi(x)$.
    Then there exists an extraction $\alpha$ and a sequence $(y_n)$ such that
    $y_n \in \psi(x_n)$ and $y_n \xrightarrow[n\longrightarrow +\infty]{} y$.
    However, $\psi$ is not u.h.c: take $x_n \xrightarrow[n\longrightarrow
    +\infty]{}1$ and the sequence $z_n := 5/3$. We have $\forall n \in \N,\; z_n
    \in \psi(x_n)$, however any subsequence of $(z_n)$ converges to $5/3 \notin
    \psi(1)$.}
    \label{fig:ex_uhc_lhc}
\end{figure}
Finally, an hemi-continuous multi-map is one that is both u.h.c. and l.h.c.:
\begin{definition}\label{def:hc} A multi-valued function $\varphi: E
    \rightrightarrows F$ from a compact metric space space $E$ to parts of a
    compact metric space space $F$ is said to be \textit{hemi-continuous} if it
    is both u.h.c. (\cref{def:uhc}) and l.h.c. (\cref{def:lhc}).
\end{definition}
We begin with technical lemmas on the hemi-continuity properties of sets of
couplings.
\begin{lemma}\label{lemma:Pi_hc} Consider $E, F$ compact metric spaces and $\nu
    \in \mathcal{P}(F)$. The multi-function
    \begin{equation}\label{eqn:def_Pi_nu}
        \Pi_\nu := \left\lbrace\begin{array}{ccc}
        \mathcal{P}(E) & \rightrightarrows & \mathcal{P}(E\times F) \\
        \mu & \mapsto & \Pi(\mu, \nu) \\ 
        \end{array}\right.
    \end{equation}
    is hemi-continuous. 
\end{lemma}

\begin{proof}
    \textbf{u.h.c..} We apply \cref{def:uhc}: introduce $\mu_n
    \xrightarrow[n\longrightarrow+\infty]{w}\mu \in \mathcal{P}(E)$ and $\pi_n
    \in \Pi(\mu_n, \nu)$. Since $\mathcal{P}(E\times F)$ is compact, we can
    introduce $\alpha$ an extraction such that $\pi_{\alpha(n)}
    \xrightarrow[n\longrightarrow+\infty]{w} \pi \in \mathcal{P}(E\times F)$. By
    continuity of marginalisation, we deduce $\pi \in \Pi(\mu, \nu)$, which
    shows that $\Pi_\nu$ is u.h.c. by definition.
    
    \textbf{l.h.c..} We consider $\W_1$, the 1-Wasserstein distance on
    $\mathcal{P}(E)$ (i.e. $\T_{d_E}$), and use the same notation for the
    1-Wasserstein distance on $\mathcal{P}(E\times F)$, with the distance
    \[d_{E\times F}((x, y), (x', y')) := \max(d_E(x, x'), d_F(y, y')),\] both of
    which metrise the weak convergence by \cite[Corollary 6.13]{villani}. We
    apply \cref{def:lhc}: take $\mu_n
    \xrightarrow[n\longrightarrow+\infty]{w}\mu \in \mathcal{P}(E)$, and let
    $\pi \in \Pi(\mu, \nu)$. Consider $(X, Y)$ two coupled random variables of
    law $\pi$, and for $n\in \N$, take $X_n$ a random variable such that $(X,
    X_n)$ is an optimal coupling for $\W_1(\mu, \mu_n)$, and let $\pi_n :=
    \mathrm{Law}(X_n, Y)$. We have
    $$\W_1(\pi, \pi_n) \leq \mathbb{E}\left[d_{E\times F}\left((X, Y), (X_n,
    Y)\right)\right] = \mathbb{E}\left[\max(d_E(X, X_n), d_F(Y, Y))\right] =
    \W_1(\mu, \mu_n),$$ then by metrisation, we get $\W_1(\mu, \mu_n)
    \xrightarrow[n\longrightarrow+\infty]{} 0$, then $\pi_n
    \xrightarrow[n\longrightarrow+\infty]{w}\pi$, concluding the proof that
    $\Pi_\nu$ is l.h.c..
\end{proof}

We can apply Berge's maximisation theorem to show that the set of
\textit{optimal} transport plans is upper hemi-continuous for a continuous cost
function:

\begin{lemma}\label{lemma:Pi_star_uhc} Consider $E, F$ compact metric spaces, a
    continuous cost $c: E \times F \longrightarrow \R_+$ and $\nu \in
    \mathcal{P}(F)$. The multi-function
    \begin{equation}\label{eqn:def_Pi_star_nu}
        [\Pi_c^*]_\nu := \left\lbrace\begin{array}{ccc}
        \mathcal{P}(E) & \rightrightarrows & \mathcal{P}(E\times F) \\
        \mu & \mapsto & \Pi_c^*(\mu, \nu) \\ 
        \end{array}\right.
    \end{equation}
    is upper hemi-continuous. 
\end{lemma}

\begin{proof}
    By compactness, the map $\pi \longmapsto \int_{E\times F} c \dd \pi$ is
    continuous, and by \cref{lemma:Pi_hc}, the multi-map $\mu \rightrightarrows
    \Pi(\mu, \nu)$ is hemi-continuous (with compact values), hence by Berge's
    maximisation theorem from \cite[Theorem 17.31]{Aliprantis1994}, the map
    $$[\Pi_c^*]_\nu : \mu \longmapsto \Pi_c^*(\mu, \nu) = \underset{\pi \in
    \Pi(\mu, \nu)}{\argmin} \int_{E\times F}c\dd \pi $$ is upper
    hemi-continuous.
\end{proof}

\begin{remark}
    \cref{lemma:Pi_star_uhc} can also be deduced from \cite[Corollary
    5.21]{villani}.
\end{remark}

\begin{remark}\label{rem:Pi_star_not_lhc} The multifunction $[\Pi_c^*]_\nu$ is
    not \textbf{lower} hemi-continuous. Indeed, take the following points of
    $\R^2$:
    $$\forall n \in \N,\; x_n := (-1, 2^{-n}),\; y_n := (1, -2^{-n}),\; x_\infty
    := (-1, 0),\; y_\infty := (1, 0),\; w := (0, 1),\; z := (0, -1), $$ and the
    following discrete measures (see \cref{fig:Pi_star_not_lhc}): 
    $$\forall n \in \N,\; \mu_n := \cfrac{1}{2}(\delta_{x_n} + \delta_{y_n}),\;
    \mu_\infty := \cfrac{1}{2}(\delta_{x_\infty} + \delta_{y_\infty}),\; \nu :=
    \cfrac{1}{2}(\delta_w + \delta_z). $$ We have $\mu_n
    \xrightarrow[n\longrightarrow +\infty]{w}\mu_\infty$, and a unique OT plan
    for the cost $c(\cdot, \cdot) := \|\cdot - \cdot\|_2^2$ between $\mu_n$ and
    $\nu$, which sends $x_n$ to $w$ and $y_n$ to $z$:
    $$\forall n \in \N,\; \Pi_{c}^*(\mu_n, \nu) = \lbrace \pi_n \rbrace, \;
    \pi_n := \cfrac{1}{2}\left(\delta_{x_n, w} + \delta_{y_n, z}\right),$$ with
    $\pi_n \xrightarrow[n\longrightarrow +\infty]{w} \pi_\infty :=
    \cfrac{1}{2}\left(\delta_{x_\infty, w} + \delta_{y_\infty, z}\right)$.
    However, the set of optimal plans between the limit $\mu_\infty$ and $\nu$
    has more than one element, since $\|x_\infty - w\|_2^2 = \|x_\infty -
    z\|_2^2$ and $\|y_\infty - w\|_2^2 = \|y_\infty - z\|_2^2$:
    $$\Pi_{c}^*(\mu, \nu) = \left\lbrace (1-t) \pi_\infty + t \pi',\; t \in [0,
    1]\right\rbrace,\; \pi' := \cfrac{1}{2}\left(\delta_{x_\infty, z} +
    \delta_{y, w}\right).$$ We conclude that there does not exist an extraction
    $\alpha$ and a sequence $(\pi'_n)$ such that $\forall n \in \N,\; \pi'_n \in
    \Pi_{c}^*(\mu_{\alpha(n)}, \nu)$ and $\pi'_n
    \xrightarrow[n\longrightarrow+\infty]{w}\pi'$.
\end{remark}
\begin{figure}[ht]
    \centering
    \includegraphics{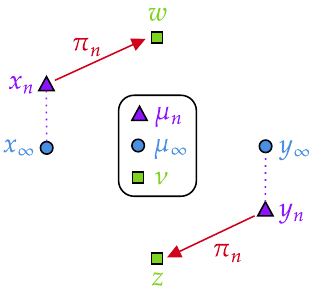}
    \caption{Counter-example from \cref{rem:Pi_star_not_lhc} showing that
    $\Pi_c^*(\cdot, \nu)$ is not lower hemi-continuous in general.}
    \label{fig:Pi_star_not_lhc}
\end{figure}

A direct corollary of \cref{lemma:Pi_star_uhc} is the upper hemi-continuity of
$\Gamma$ and $G$. For notational convenience, we introduce $\Z := \X \times \Y_1
\times \cdots \times \Y_K$.

\begin{prop}\label{prop:G_uhc} The multi-map
    $$\Gamma := \left\lbrace\begin{array}{ccc} \mathcal{P}(\X) &
        \rightrightarrows & \mathcal{P}(\Z) \\
        \mu & \mapsto & \Gamma(\mu) \\ 
        \end{array}\right. $$ where $\Gamma(\mu)$ is defined in
    \cref{eqn:def_Gamma_mu} and $G$ defined in \cref{eqn:def_G} are upper
    hemi-continuous (and compact-valued).
\end{prop}

\begin{proof}
    Let $\mu \in \mathcal{P}(\X)$. To show that $G(\mu)$ and $\Gamma(\mu)$ are
    compact, it suffices to show that $\Gamma(\mu)$ is closed, since
    $\mathcal{P}(\Z)$ is compact, and $G(\mu) = B\#\Gamma(\mu)$ with $B$
    continuous by \cref{lemma:B_C0}. Take $(\gamma_n) \in \Gamma(\mu)^\N$ such
    that $\gamma_n \xrightarrow[n\longrightarrow+\infty]{}\gamma \in
    \mathcal{P}(\Z)$. We show that $\gamma \in \Gamma(\mu)$. For $k\in
    \llbracket 1, K \rrbracket$ and $n \in \N$, we have $\gamma_n \in
    \Gamma(\mu)$, hence $[\gamma_n]_{0,k} \in \Pi_{c_k}^*(\mu, \nu_k)$. By
    continuity of marginalisation, we deduce that $\gamma \in \Pi(\mu, \nu_1,
    \cdots, \nu_K)$. By continuity of $\pi \longmapsto \int_{\X\times\Y_k}c_k
    \dd \pi$ (which holds by compactness), we deduce that $\gamma_{0,k} \in
    \Pi_{c_k}^*(\mu, \nu_k)$, hence $\gamma \in \Gamma(\mu)$.
    
    For the u.h.c. of $\Gamma$, take a sequence $(\mu_n) \in \mathcal{P}(\X)^\N$
    such that $\mu_n \xrightarrow[n\longrightarrow+\infty]{w}\mu\in
    \mathcal{P}(\X)$, and take a sequence $(\gamma_n) \in \mathcal{P}(\Z)^\N$,
    with $\gamma_n \in \Gamma(\mu_n)$. Since $\gamma_n \in \mathcal{P}(\Z)$
    which is compact, take $\alpha$ an extraction such that $\gamma_{\alpha(n)}
    \xrightarrow[n\longrightarrow+\infty]{w} \gamma \in \mathcal{P}(\Z)$. We
    will show that $\gamma \in \Gamma(\mu)$.
    
    Start with $k := 1$. For $n \in \N$, we have $\gamma_{\alpha(n)} \in
    \Gamma(\mu_{\alpha(n)})$, hence $\pi_{\alpha(n)}^{(1)} :=
    [\gamma_{\alpha(n)}]_{0,1} \in \Pi_{c_1}^*(\mu_{\alpha(n)}, \nu_1)$. By
    \cref{lemma:Pi_star_uhc}, the map $\mu \longmapsto \Pi_{c_1}^*(\mu, \nu_1)$
    is u.h.c., hence by definition, since $\mu_{\alpha(n)}
    \xrightarrow[n\longrightarrow+\infty]{w}\mu\in \mathcal{P}(\X)$ and
    $\pi_{\alpha(n)}^{(1)} \in \Pi_{c_1}^*(\mu_{\alpha(n)}, \nu_1)$, there
    exists an extraction $\alpha_1$ such that
    $\pi_{\alpha\circ\alpha_1(n)}^{(1)}
    \xrightarrow[n\longrightarrow+\infty]{w}\pi^{(1)} \in \Pi_{c_1}^*(\mu,
    \nu_1)$. 
    
    Continuing this method for $k \in \llbracket 2, K \rrbracket$ with
    successive sub-extractions $\alpha_k$, setting $\beta := \alpha \circ
    \alpha_1 \circ \cdots \circ \alpha_K$, we have for any $k\in\llbracket 1, K
    \rrbracket,\; [\gamma_{\beta(n)}]_{0, k} = \pi_{\beta(n)}^{(k)}
    \xrightarrow[n\longrightarrow+\infty]{w} \pi^{(k)} \in \Pi_{c_k}^*(\mu,
    \nu_k)$. The continuity of marginalisation implies $\gamma_{0,k} =
    \pi^{(k)}$, and in turn shows that $\gamma \in \Gamma(\mu)$, concluding that
    $\Gamma$ is u.h.c.
    
    For $G$, the fact that $G(\mu) = B\#\Gamma(\mu)$ and the continuity of $B$
    prove that $G$ is u.h.c. using the u.h.c. of $\Gamma$ by \cite[Theorem
    17.23]{Aliprantis1994}.
\end{proof}

In order to study the energy of iterates of $G$, we first require a technical
result on the error of sub-optimal ground barycentres for $B$. We introduce a
radius constant $R := \underset{(x, Y) \in \X\times\Y}{\max}\ d_\X(x, B(Y))$,
Note that $R>0$. Indeed, let $a,b \in \mathcal{X}$ be such that
$d_\mathcal{X}(a,b) = \mathrm{diam}(\mathcal{X})$. Then, for all $Y\in
\mathcal{Y}$, it holds
\[
2R \geq d_\mathcal{X}(a,B(Y)) + d_\mathcal{X}(B(Y), b) \geq d_\mathcal{X}(a,b) = \mathrm{diam}(\mathcal{X}),
\]
and so $R \geq \mathrm{diam}(\mathcal{X})/2>0.$
\begin{lemma}\label{lemma:costs_delta} There exists a function $\delta = \eta
    \circ d_\X$, with $\eta: [0, R] \longrightarrow \R_+$ lower-semi-continuous,
    non-decreasing and verifying $\eta(s)=0 \Longleftrightarrow s = 0$, such
    that
    \begin{equation}\label{eqn:costs_delta}
        \forall (x, Y) \in \X \times \Y,\; C(x, Y) \geq C(B(Y), Y) + \delta(x, B(Y)).
    \end{equation}
\end{lemma}
\begin{proof}
    \step{1}{Definition of $\eta$.} First, for $(x, Y) \in \X \times \Y$, let
    $\Delta(x, Y) := C(x, Y) - C(B(Y), Y)$. By definition of $B$, $\Delta(x, Y)
    \geq 0$, and $\Delta(x, Y) = 0 \Longleftrightarrow x = B(Y)$. By assumption,
    $B$ and $C$ are continuous, which implies that $\Delta$ is also continuous.
    
    We now introduce $S := \max_{(x, Y)\in \X\times \Y} \Delta(x, Y)$. $R>0$
    ensures $S>0$. Define now the function $\eta$:
    \begin{equation}
        \eta := \app{[0, R]}{[0, S]}{u}{\underset{(x, Y) \in \X\times\Y}{\min}\ \left\{\Delta(x, Y)\ :\ d_\X(x, B(Y)) \geq u\right\}}.
    \end{equation}
    We show that for $u\in [0, R]$, the infimum is attained. First, let $f :=
    (x, Y) \longmapsto d_\X(x, B(Y))$, we remark that 
    $$\forall (x, Y) \in \X\times \Y,\; d_\X(x, B(Y)) \geq u \Longleftrightarrow
    (x, Y) \in f^{-1}([u, R]).$$ By continuity of $f$ and compactness of
    $\X\times\Y$, $\K_u := f^{-1}([u, R])$ is a compact subset of $\X\times\Y$.
    $\K_u$ is not empty since there exists $(x_R, Y_R) \in \X\times\Y$ such that
    $d_\X(x_R, B(Y_R)) = R$ (by continuity, compactness and definition of $R$).
    
    \step{2}{Proof of \cref{eqn:costs_delta}.} Let $(x, Y) \in \X\times \Y$, and
    $u := d_\X(x, B(Y))$. By definition, $(x, Y) \in \K_u$, hence $\eta(u) \leq
    \Delta(x, Y)$, which is equivalent to \cref{eqn:costs_delta}.
    
    \step{3}{Lower semi-continuity of $\eta$}. Let $u_n \xrightarrow[n
    \longrightarrow +\infty]{} u \in [0, R]$, and for $n \in \N$ introduce $
    (x_n, Y_n) \in \K_{u_n}$ such that $\eta(u_n) = \Delta(x_n, Y_n)$. Since
    $(\eta(u_n)) \in [0, S]^\N$, consider an extraction $\alpha$ such that
    $\eta(u_{\alpha(n)}) \xrightarrow[n \longrightarrow +\infty]{} a_{\alpha}
    \in [0, S]$. By compactness of $\X\times \Y$, we can extract from
    $(x_{\alpha(n)}, Y_{\alpha(n)})_n$ a subsequence such that $(x_{\alpha \circ
    \beta (n)}, Y_{\alpha\circ \beta (n)})
    \xrightarrow[n\longrightarrow+\infty]{}(x_{\alpha, \beta}, Y_{\alpha,
    \beta}) \in \X\times \Y$. By construction of the sequence $(x_n, Y_n)_n$, we
    have 
    \begin{equation}\label{eqn:eta_lsc_subsequence_ineq}
        \forall n \in \N, \; d_\X(x_{\alpha\circ\beta(n)},
    B(Y_{\alpha\circ\beta(n)})) \geq u_{\alpha\circ\beta(n)},
    \end{equation}
    since $(x_{\alpha\circ\beta(n)}, Y_{\alpha\circ\beta(n)}) \in
    \K_{u_{\alpha\circ\beta(n)}}$. Taking the limit in
    \cref{eqn:eta_lsc_subsequence_ineq} yields  $d_\X(x_{\alpha, \beta},
    B(Y_{\alpha, \beta})) \geq u,$ by continuity of $B$, \cref{lemma:B_C0}. This
    shows that $(x_{\alpha, \beta}, Y_{\alpha, \beta}) \in \K_u$, hence $\eta(u)
    \leq \Delta(x_{\alpha, \beta}, Y_{\alpha, \beta})$. However, by continuity
    of $\Delta$, and since $\Delta(x_{\alpha(n)}, Y_{\alpha(n)})
    \xrightarrow[n\longrightarrow+\infty]{}a_\alpha$, it follows that
    $\Delta(x_{\alpha, \beta}, Y_{\alpha, \beta}) = a_\alpha$. Since the
    subsequential limit $a_\alpha$ was chosen arbitrarily, we conclude that
    $\eta(u) \leq \underset{n\longrightarrow+\infty}{\liminf}\eta(u_n)$, hence
    $\eta$ is lower semi-continuous.
    
    \step{4}{$\eta$ is non-decreasing.} Let $0 \leq u \leq v \leq R$, we have
    $\K_v \subset \K_u$, hence
    $$\eta(u) = \underset{(x, Y) \in \K_u}{\min}\ \Delta(x, Y) \leq
    \underset{(x, Y) \in \K_v}{\min}\ \Delta(x, Y) = \eta(v). $$
    \step{5}{Separation property.} Let $u \in [0, R]$ such that $\eta(u) = 0$.
    This implies that there exists $(x, Y) \in \X\times \Y$ such that $\Delta(x,
    Y) = 0$ and $d_\X(x, B(Y))\geq u$. Now by Step 1 this implies $x = B(Y)$,
    thus $d_\X(x, B(Y)) = 0$ and finally $u=0$.
\end{proof}
\cref{lemma:costs_delta} is a generalisation of the following elementary
Euclidean property in $\R^d$ for the cost $\|\cdot-\cdot\|_2^2$, for which
$B(y_1, \cdots, y_K) = \sum_{k=1}^K\lambda_ky_k$ verifies the following
identity:
$$\forall x \in \R^d,\; \forall (y_1, \cdots, y_K) \in (\R^d)^K,\; \oll{x} :=
\sum_{k=1}^K\lambda_k y_k:\; \sum_{k=1}^K\lambda_k\|x-y_k\|_2^2 =
\sum_{k=1}^K\lambda_k \|\oll{x}-y_k\|_2^2 + \|x-\oll{x}\|_2^2.$$
Given the inequality in \cref{eqn:costs_delta}, we can now find an informative
inequality between $V(\oll{\mu})$ and $V(\mu)$ for any $\oll{\mu}\in G(\mu)$.

\begin{prop}\label{prop:decreasing_iterates} Let $\mu \in \mathcal{P}(\X)$ and
    $\oll{\mu} \in G(\mu)$. Then $V(\mu) \geq V(\oll{\mu}) + \T_\delta(\mu,
    \oll{\mu})$. If $\mu^*$ is a barycentre, then $G(\mu^*) = \{\mu^*\}$.
\end{prop}
\begin{proof}
    Let $\oll{\mu} = B\#\gamma \in G(\mu)$ with $\gamma \in \Gamma(\mu)$. By
    definition of $\T_{c_k}$ and by optimality of the bi-marginals
    $\gamma_{0,k}$ of $\gamma$:
    \begin{align}
        \sum_{k=1}^K \T_{c_k}(\mu, \nu_k) &= \int_{\X\times\Y} C(x, Y) \dd \gamma(x, Y) \label{eqn:decreasing_iterates_def_Tck}\\
        &\geq \int_{\X\times\Y} (C(B(Y), Y) + \delta(x, B(Y))) \dd \gamma(x, Y)  \label{eqn:decreasing_iterates_B_ineq} \\
        &\geq \sum_{k=1}^K\T_{c_k}(B\#\gamma, \nu_k) + \T_\delta(\mu, B\#\gamma)\label{eqn:decreasing_iterates_suboptimal_coupling} \\ 
        &= V(\oll{\mu}) + \T_\delta(\mu, \oll{\mu}).
    \end{align}
    The inequality in \cref{eqn:decreasing_iterates_B_ineq} comes from
    \cref{lemma:costs_delta}, and the inequality in
    \cref{eqn:decreasing_iterates_suboptimal_coupling} comes from the definition
    of $\Gamma(\mu)$ (\cref{eqn:def_Gamma_mu}), which allows us to write for $k
    \in \llbracket 1, K\rrbracket$:
    \begin{align*}
        \int_{\X\times\Y}c_k(B(Y), y_k) \dd \gamma(x, Y) &= \int_{\X\times\Y_k}c_k \dd \pi_k,
    \end{align*}
    where we introduce the coupling $\pi_k := (B, P_k)\#[\gamma_{1, \cdots,
    K}]$, with $P_k(y_1, \cdots, y_K) = y_k$. The first marginal of $\pi$ is
    $B\#[\gamma_{1, \cdots, K}]$ (which we write $B\#\gamma$ for legibility),
    and the second marginal is $\nu_k$. Similarly, 
    $$\int_{\X\times\Y}\delta(x, B(Y))\dd\gamma(x, Y) = \int_{\X\times\X}\delta
    \dd [(I, B)\#\gamma] \geq \T_\delta(\mu, B\#\gamma).$$
    
    If $\mu^*$ is a barycentre, then by definition for any $\oll{\mu} \in
    G(\mu)$, we have $V(\oll{\mu})\geq V(\mu^*)$, thus
    \cref{eqn:decreasing_iterates_B_ineq,eqn:decreasing_iterates_suboptimal_coupling}
    are equalities, and $\T_\delta(\mu^*, \oll{\mu})=0$. By
    \cref{lemma:costs_delta,lemma:Tc_regularity}, the cost $\delta$ guarantees
    the separation property of the transport cost $\T_\delta$, hence $\mu^* =
    \oll{\mu}$.
\end{proof}

Applying \cref{prop:decreasing_iterates} to the $\W_2$ case for absolutely
continuous measures yields \cite[Proposition 4.3]{alvarez2016fixed}, wherein the
cost $\T_\delta$ is simply $\W_2^2$. This decrease was also studied by
\cite[Proposition 4.4]{lindheim2023simple} in the discrete setting $\W_p^p$.

The inequality in \cref{prop:decreasing_iterates} shows that the amount of
decrease in the energy between two iterations is lower-bounded by a transport
discrepancy $\T_\delta$ (we remind that in the squared-Euclidean case,
$\T_\delta = \W_2^2$). We can now show convergence of iterates of $G$, in the
sense that any weakly converging subsequence converges towards a fixed point of
$G$.

\begin{theorem}\label{thm:fixed_point_iterates_cv} For any $\mu_0 \in
    \mathcal{P}(\X)$, let $(\mu_t)$ verifying $\mu_{t+1} \in G(\mu_t)$. Then
    $(\mu_t)$ has converging subsequences, and any weakly converging subsequence
    necessarily converges towards a $\mu \in \mathcal{P}(\X)$ such that $\mu \in
    G(\mu)$.
\end{theorem}
\begin{proof}
    Fix a sequence $(\mu_t)$ such that $\mu_{t+1} \in G(\mu_t)$ and write
    $\mu_{t+1} = B\#[\gamma_t]_{1,\cdots, K}$ with $\gamma_t \in \Gamma(\mu_t)$.
    Since $\X$ is compact, the space $\mathcal{P}(\X)$ is also compact, and so
    the sequence $(\mu_t)$ is tight.
    Consider an extraction $\alpha$ such that $\mu_{\alpha(t)}
    \xrightarrow[t\longrightarrow +\infty]{w} \mu \in \mathcal{P}(\X)$. By
    u.h.c. of $\Gamma$ (\cref{prop:G_uhc}), there exists an extraction $\beta$
    such that $\gamma_{\alpha\circ\beta (t)}
    \xrightarrow[t\longrightarrow+\infty]{w}\gamma \in \Gamma(\mu)$.
    
    By \cref{prop:decreasing_iterates}, the sequence $(V(\mu_t))$ is
    non-increasing and non-negative, hence it is convergent, and as a result,
    $\underset{t\longrightarrow+\infty}{\lim}\left[V(\mu_{\alpha\circ\beta(t)})
    - V(\mu_{\alpha\circ\beta(t)+1})\right]=0$. Using the lower-bound in
    \cref{prop:decreasing_iterates} we obtain:
    $$\forall t \in \N,\; 0\leq \T_\delta(\mu_{\alpha\circ\beta(t)},
    \mu_{\alpha\circ\beta(t)+1}) \leq V(\mu_{\alpha\circ\beta(t)}) -
    V(\mu_{\alpha\circ\beta(t)+1}), $$ and take the limit inferior:
    \begin{equation}\label{eqn:fp_cv_plus_one_ineq}
        0\leq \underset{t\longrightarrow+\infty}{\liminf}\
        \T_{\delta}(\mu_{\alpha\circ\beta(t)}, \mu_{\alpha\circ\beta(t)+1}) \leq 0.
    \end{equation}
    We remind that $(\mu_{\alpha\circ\beta(t) + 1})_t$ is a sequence in
    $\mathcal{P}(\X)$ which is compact, and take $\rho\in \mathcal{P}(\X)$ a
    subsequential limit of $(\mu_{\alpha\circ\beta(t) + 1})_t$. By
    lower-semi-continuity of $\T_\delta$ (which holds by applying
    \cref{lemma:Tc_regularity} item 1) with \cref{lemma:costs_delta}),
    \cref{eqn:fp_cv_plus_one_ineq} provides $\T_\delta(\mu, \rho) = 0$. By
    \cref{lemma:Tc_regularity} item 3), we obtain that $\rho=\mu$, thus any
    subsequential limit of $(\mu_{\alpha\circ\beta(t) + 1})_t$ is $\mu$, which
    proves that it converges weakly to $\mu$.
    
    Writing abusively $B\#\gamma$ for $B\#\gamma_{1, \cdots, K}$, we conclude:
    $$\begin{array}{ccc} \mu_{\alpha\circ\beta(t)+1} &
        \xrightarrow[t\longrightarrow +\infty]{w} & \mu \\
        \verteq & & \verteq \\
        B\#\gamma_{\alpha\circ\beta(t)} & \xrightarrow[t\longrightarrow
        +\infty]{w} & B\#\gamma \\
    \end{array} $$ hence we have found $\gamma \in \Gamma(\mu)$ such that $\mu =
    B\#\gamma$, proving $\mu \in G(\mu)$.
\end{proof}

\begin{remark}
    Fixed-points of $G$ may not be unique and may not be barycentres, as shown
    the the following example. Take the following measures:
    $$\mu := \frac{1}{2}\left(\delta_{(0, 1)} + \delta_{(0, -1)}\right),\; \nu_1
    := \frac{1}{2}\left(\delta_{(-2, 1)} + \delta_{(2, -1)}\right),\; \nu_2 :=
    \frac{1}{2}\left(\delta_{(-2, -1)} + \delta_{(2, 1)}\right).$$ Between $\mu$
    and $\nu_k$, the unique OT plan for the squared-Euclidean cost is given by a
    permutation, with 
    $$\pi_1^* = \frac{1}{2}\left(\delta_{(0, 1)\otimes(-2, 1)} + \delta_{(0,
    -1)\otimes(2, -1)}\right),$$ and likewise for $\pi_2^*$. The next iterate of
    $G$ and $H$ are both equal to $\mu$ itself, which is distinct from the
    unique barycenter $\mu^* = \frac{1}{2}(\delta_{(-2, 0)} + \delta_{(2, 0)})$.
    We show this example in \cref{fig:ex_fp_G_not_bar}.
\end{remark}

\begin{figure}[ht]
    \begin{center}
        \includegraphics[width=0.4\linewidth]{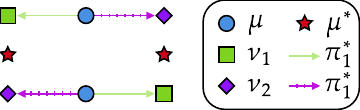}
    \end{center}
    \caption{Example showing a non-barycentre measure $\mu$ which is a
    fixed-point of $G$ and $H$.}
    \label{fig:ex_fp_G_not_bar}
\end{figure} 

\subsection{Expression of the Iterates when the Plans are
Maps}\label{sec:bar_proj}

In some cases, the plans introduced in $\Gamma(\mu)$ (\cref{eqn:def_Gamma_mu})
are induced by maps, which is to say that they are each supported on a set of
the form $(x, T_k(x))$. This is the case in the specific setting chosen by
\cite{alvarez2016fixed}, which is to say that all measures are absolutely
continuous on $\R^d$ and the costs are all $c(x,y) = \|x-y\|_2^2$. By Brenier's
Theorem (as stated in \cite[Theorem 1.22]{santambrogio2015optimal}, for
example), this implies that optimal transport couplings are supported on the
graph of a map. This property holds under the weaker condition that the costs
verify the Twist condition (see \cite[Theorem 10.28]{villani} for example). In
this case, each set optimal transport plans $\Pi_{c_k}^*(\mu, \nu_k)$ is
composed of one element $(I, T_k)\#\mu$, and as a result, the expression of
$G(\mu)$ becomes substantially simpler, namely $G(\mu) = \{B \circ (T_1, \cdots,
T_K) \#\mu\}$. In the linearisation interpretation
(\cref{fig:fixed_point_linearisation}), this expression can be understood as
taking the ground barycentre of the maps $T_k$ using the ground map $B$. 

Drawing inspiration from this observation, we can define an alternative
iteration consisting in choosing a map $T_k$ as the barycentric projection of
the coupling $\gamma_{0,k}\in \Pi_{c_k}^*(\mu, \nu_k)$ for $\gamma \in
\Gamma(\mu)$: see \cref{def:barycentric_projection} and
\cref{fig:barycentric_projection}.
\begin{definition}\label{def:barycentric_projection} The \textbf{barycentric
    projection} of a coupling $\pi \in \Pi(\mu, \nu)$ for $\mu\in
    \mathcal{P}(E)$ and $\nu \in \mathcal{P}(F)$ is the map $\oll{\pi}: E
    \longrightarrow F$, which is defined for $\mu$-almost-every $x\in E$ as:
    $$\oll{\pi}(x) = \int_{F}y \pi^x(\dd y),  $$ where we wrote the
    disintegration $\pi(\dd x, \dd y) = \mu(\dd x)\pi^x(\dd y)$. In terms of
    random variables, one may write this expression as:
    $$\oll{\pi}(x) = \underset{(X, Y) \sim \pi}{\mathbb{E}}[Y\ |\ X = x]. $$
\end{definition}
\begin{figure}[ht]
    \centering
    \includegraphics[width=.5\linewidth]{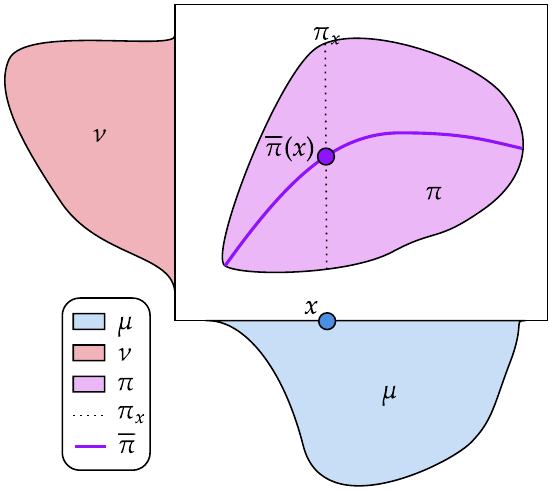}
    \caption{Illustration of a barycentric projection. The disintegration of the coupling $\pi$ with respect to its first marginal $\mu$ at $x$ is the measure $\pi_x$ concentrated on the dotted line. The barycentric projection of $\pi$ evaluated at $x$ is the mean of the measure $\pi_x$.}
    \label{fig:barycentric_projection}
\end{figure}
Note that for this expression to be well-defined, the target space $F$ must be a
\textit{convex space}, i.e. a space where one may define convex combinations of
points (or, more precisely, expectations of probability measures). In the case
$\X = \Y_1 = \cdots = \Y_K$, a meaningful choice of convex combination is the
ground barycentre $B$. We can apply this barycentric projection idea to define
an alternate multi-mapping $H: \mathcal{P}(\X) \rightrightarrows
\mathcal{P}(\X)$:
\begin{equation}\label{eqn:def_H}
    \forall \mu \in \mathcal{P}(\X),\; H(\mu) := \left\{B \circ (\oll{\gamma_{0,1}}, \cdots, \oll{\gamma_{0,K}}) \# \mu,\; \gamma \in \Gamma(\mu) \right\}.
\end{equation}
In general, for $\pi \in \Pi(\mu, \nu),\; \oll{\pi}\#\mu \neq \nu$, hence one
does not necessarily have $\forall \widetilde{\mu} \in H(\mu),\;
V(\widetilde{\mu}) \leq V(\mu)$. However, if each $\Pi_{c_k}^*(\mu, \nu_k)$ are
composed of plans supported by maps, then $H(\mu) = G(\mu)$. In the case of
discrete measures and for the squared Euclidean cost, the iterations of $H$
correspond to the approach proposed in \cite[Algorithm 2]{cuturi14fast}.

\subsection{The Particular Case of Conditionally Independent Couplings}\label{sec:fp_cond_indep}

In \cref{eqn:def_Gamma_mu}, we chose all possible multi-couplings with optimal
bi-marginals. It is possible to restrict the set of couplings to the smaller set
of multi-couplings with conditionally independent marginals, i.e.
multi-couplings $\gamma \in \Pi(\mu, \nu_1, \cdots, \nu_K)$ such that there
exists $\pi_k \in \Pi_{c_k}^*(\mu, \nu_k)$ for $k\in \llbracket 1, K \rrbracket$
such that $\gamma_{0, k} = \pi_k$ and specifically:
$$\gamma(\dd x, \dd y_1, \cdots, \dd y_K) := \mu(\dd x)\ \pi_1^x(\dd y_1) \cdots
\pi_K^x(\dd y_K), $$ as in \cref{eqn:glue}. In terms of random variables, this
corresponds to the choice of $(X, Y_1, \cdots, Y_K)$ such that $(X, Y_k) \sim
\pi_k$ and conditionally to $X$, the variables $Y_1, \cdots, Y_K$ are
independent. We denote by $\Gamma_{\otimes}(\mu)$ the set of such couplings, and
consider the associated multi-map $G_\otimes := B\#\Gamma_\otimes$ as in
\cref{eqn:def_G}. It is clear that $\forall \mu \in \mathcal{P}(\X),\;
G_\otimes(\mu)\subset G(\mu)$. In particular, this implies subsequential
convergence converges of iterates $\mu_{t+1}=G_\otimes(\mu_t)$ to a fixed-point
of $G$. In \cref{prop:fp_Gotimes}, we show that the convergence is to a
fixed-point of $G_\otimes$ in the discrete case (measures with finite support).
First, we emphasise that with a discrete initialisation measure and discrete
measures $(\nu_k)$, the support of the sequence $(\mu_t)$ is finite and always
contained in:
$$\left\{B(y_1, \cdots, y_K),\; \forall k \in \llbracket 1, K \rrbracket,\; y_k
\in \supp(\nu_k) \right\},$$
which ensures that iterates remain discrete.
\begin{prop}\label{prop:fp_Gotimes} Take $\mu_0 \in \mathcal{P}(\X)$ a discrete
    measure and $\nu_1, \cdots, \nu_k \in \mathcal{P}(\Y_1)\times \cdots \times
    \mathcal{P}(\Y_K)$ discrete measures. Then any sub-sequential limit $\mu \in
    \mathcal{P}(\X)$ of the sequence $(\mu_t)$ defined by $\mu_{t+1} \in
    G_\otimes(\mu_t)$ verifies $\mu\in G_\otimes(\mu)$.
\end{prop}
\begin{proof}
    We follow a technique used in the proof of \cite[Theorem
    9.6]{gozlan2017kantorovich}, specifically page 65. As commented before the
    statement of the result, the sequence $(\mu_t)$ remains discrete. Write for
    $t \in \N,\; \mu_{t+1} = B\#\gamma_t$ with $\gamma_t
    \in\Gamma_\otimes(\mu_t)$, and take an extraction $\alpha$ such that
    $\mu_{\alpha(t)}\xrightarrow[t\longrightarrow +\infty]{w} \mu$. As done in
    \cref{thm:fixed_point_iterates_cv}, the u.h.c. property of $\Gamma$ allows
    us to extract a subsequence $\beta$ such that $\gamma_{\alpha\circ\beta(t)}
    \xrightarrow[t\longrightarrow +\infty]{w} \gamma \in \Gamma(\mu)$, since we
    have the (point-wise) inclusion $\Gamma_\otimes \subset \Gamma$. As shown in
    the proof of \cref{thm:fixed_point_iterates_cv}, the sequence
    $(\mu_{\alpha\circ\beta(t)})_t$ weakly converges to $\mu$, hence we now want
    to show that $\gamma \in \Gamma_\otimes(\mu)$, which would allow to conclude
    $\mu \in G_\otimes(\mu)$.

    For $t\in \N$ and $k \in \llbracket 1, K \rrbracket$, introduce
    $\pi_{\alpha\circ\beta(t)}^{(k)} := [\gamma_{\alpha\circ\beta(t)}]_{0,k}$,
    and its disintegration with respect to $\mu_{\alpha\circ\beta(t)}$ as
    $$ \pi_{\alpha\circ\beta(t)}^{(k)}(\dd x, \dd y_k) =
    \mu_{\alpha\circ\beta(t)}(\dd x)\ [\pi_{\alpha\circ\beta(t)}^{(k)}]^x(\dd
    y_k).$$ As argued above the statement of the proposition, the sequence
    $(\mu_t)_t$ remains discrete with a support contained in $B(\prod_k
    \supp(\nu_k))$ and thus $(\gamma_t)$ also remains discrete, and its first
    marginal $\mu_t$ has a finite support of size at most $n := \prod_k
    \#\supp(\nu_k)$ on fixed points $(x_1, \cdots, x_n)$. For simplicity, we
    will see the measures $(\mu_t)$ as supported on $\X_n := \{x_1, \cdots,
    x_n\}$ with possibly zero mass at some of these points, and in such cases,
    we define $[\pi_{\alpha\circ\beta(t)}^{(k)}]^x$ as the null measure
    $\mathcal{M}_0$. Since $\gamma_{\alpha\circ\beta(t)} \in
    \Gamma_\otimes(\mu_{\alpha\circ\beta(t)})$, by definition we can write its
    disintegration with respect to $\mu_{\alpha\circ\beta(t)}$ as:
    $$\gamma_{\alpha\circ\beta(t)}(\dd x, \dd y_1, \cdots , \dd y_K) =
    \mu_{\alpha\circ\beta(t)}(\dd x)\  [\pi_{\alpha\circ\beta(t)}^{(1)}]^x(\dd
    y_1) \cdots [\pi_{\alpha\circ\beta(t)}^{(K)}]^x(\dd y_K).$$   
    For $i \in \llbracket 1, n \rrbracket$ and $k \in \llbracket 1, K
    \rrbracket,$ there exists an extraction $\chi_{i,k}$ such that the sequence
    $\left([\pi_{\alpha \circ \beta\circ\chi_{i,k}(t)}^{(k)}]^{x_i}\right)_{t\in
    \N}$ converges weakly to a $[\pi^{(k)}]^{x_i} \in
    \mathcal{P}(\X)\times\{\mathcal{M}_0\}$. We choose the extractions as
    successive sub-extractions, such that $\chi_{1, 2}$ is a sub-extraction of
    $\chi_{1, 1}$, until $\chi_{n, K}$ which is a sub-extraction of all previous
    extractions. We then define $\chi := \chi_{n, K}$. The extraction $\chi$ is
    such that for $i \in \llbracket 1, n \rrbracket$ and $k \in \llbracket 1, K
    \rrbracket$, the sequence
    $\left([\pi_{\alpha\circ\beta\circ\chi(t)}^{(k)}]^{x_i}\right)_{t\in \N}$
    converges weakly to $[\pi^{(k)}]^{x_i}$. By verifying against test
    functions, we deduce the following disintegration holds for $\gamma$:
    $$\gamma(\dd x, \dd y_1, \cdots, \dd y_K) = \mu(\dd x)\ [\pi^{(1)}]^x(\dd
    y_1) \cdots [\pi^{(K)}]^x(\dd y_K),$$ which shows that $\gamma \in
    \Gamma_\otimes(\mu)$, and thus $\mu \in G_\otimes(\mu)$.
\end{proof}
\begin{remark}
    The proof of \cref{prop:fp_Gotimes} can also be written for discrete
    measures with at-most-countable supports through a diagonal extraction
    argument, we kept to finite supports for legibility.
\end{remark}
\section{Focus on the Discrete Case}\label{sec:discrete}

In this section, we will formulate the fixed-point algorithm in the discrete
case, and discuss some algorithmic aspects.

\subsection{Discrete Expression and Algorithms}\label{sec:discrete_expressions}

Consider discrete measures $\nu_k := \sum_{i=1}^{n_k}b_{k, i}\delta_{y_{k, i}}
\in \mathcal{P}(\R^{d_k})$ where $\forall k \in \llbracket 1, K \rrbracket,\;
\forall i \in \llbracket 1, n_k \rrbracket,\; y_{k,i}\in \R^{d_k}$. We stack the
support of $\nu_k$ into $Y_k \in \R^{n_k\times d_k}$ such that $[Y_k]_{i,\cdot}
= y_{k, i}$, and similarly introduce $b_k := (b_{k, i})_{i=1}^{n_k}\in
\Delta_{n_k}$. 

First, our objective is to re-write the iteration \cref{eqn:def_G} in this
discrete setting, with an initial measure $\mu = \sum_{i=1}^na_i\delta_{x_i}\in
\mathcal{P}(\R^d)$. For each $k$, we choose $\pi_k \in \R_+^{n\times n_k}$ an
optimal transport plan, which is to say a solution of the Kantorovich linear
program:
\begin{equation*}
    \underset{\Pi(a, b_k)}{\argmin}\ \sum_{i=1}^n\sum_{j=1}^{n_k}c_k(x_i, y_{k, j})\pi_{i,j},
\end{equation*}
where $\Pi(a, b_k) := \left\{\pi\in \R_+^{n\times n_k}\ :\  \pi \mathbf{1} =
a,\; \pi^T \mathbf{1} = b_k\right\}.$ Seeing multi-couplings $\gamma \in
\Gamma(\mu)$ as a tensors $\gamma \in \R^{n\times n_1 \times \cdots,
n_k}$, the discrete expression of $G$ reads:
\begin{equation}\label{eqn:G_discrete}
    G(\mu) = \left\{\sum_{j_1, \cdots, j_K}\left(\sum_{i=1}^n
    \gamma_{1, j_1, \cdots, j_K}\right)\delta\left(
    B(y_{1, j_1}, \cdots, y_{K, j_K})\right),\; 
    \gamma \in \Gamma(\mu)\right\}.
\end{equation}
A visualisation of \cref{eqn:G_discrete} using the multi-coupling from
\cref{eqn:glue} is provided in \cref{fig:G_discrete}.
\begin{figure}[ht]
    \begin{center}
        \includegraphics[width=0.6\linewidth]{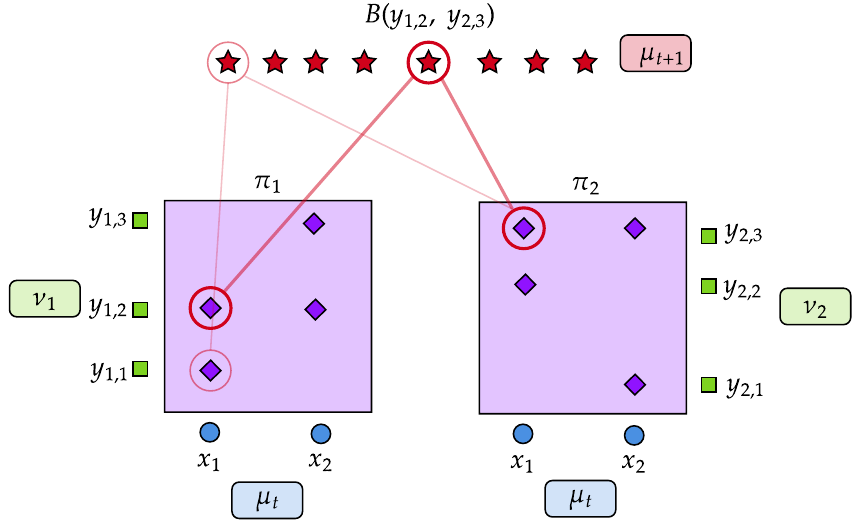}
    \end{center}
    \caption{Visual explanation of the discrete fixed-point iteration $G$. For
    each point $x_i$ in the support of $\mu_t$, we look at all the points
    $(y_{1, j_1}, \cdots, y_{K, j_K})$ which are assigned from $x_i$ by the
    multi-coupling $\gamma$, then the ground barycentre $B(y_{1, j_1}, \cdots,
    y_{K, j_K})$ is taken on all these tuples, with weights given by the
    multi-coupling.}
    \label{fig:G_discrete}
\end{figure}

As in \cite{beier2025tangential}, we can use a generalisation of the
North-West Corner (NWC) method to compute $\gamma \in \Gamma(\mu)$ with
prescribed bi-marginals $\pi_k \in \Pi(a, b_k)$. In \cref{alg:NWC}, we present
the NWC strategy. The idea is to fill the entries of $\gamma$ greedily using
entries $\pi^{(k)}_{i, j_k}$ for increasing $i, j_1, \cdots, j_K$ (see
\cite[Section 3.4.2]{computational_ot} for a presentation of the method in the
standard setting).

\begin{figure}[ht]
	\centering
	\begin{minipage}{.95\linewidth}
		\begin{algorithm}[H]
			\SetAlgoLined \KwData{For $k \in \llbracket 1, K \rrbracket$,
			transport plan $\pi_k \in \Pi(a, b_k)$, with $a \in \Delta_n$ and
			$b_k \in \Delta_{n_k}$.}
			\KwResult{Gluing $\NWC(\pi_1, \cdots, \pi_K) = \gamma \in 
            \Pi(a, b_1, \cdots, b_K)$ such that each $\gamma_{0,k} = \pi_k$.}
			\textbf{Initialisation:} $\gamma = 0_{n\times n_1\times \cdots
			\times n_K}$ and for $k \in \llbracket 1, K \rrbracket,\;
            P_k = \pi_k$.\\
			\For{$i \in \llbracket 1, n \rrbracket$}{ 
                Set $(j_1, \cdots, j_K) = (1, \cdots, 1)$ and $u = a_i$;\\
                \While{$u>0$}{
                    Compute $v = \min (P_{i, j_1}^{(1)}, \cdots, P_{i,
                    j_K}^{(K)})$;\\
                    Assign $\gamma_{i, j_1, \cdots, j_K} = v$ and decrease $u
                    \gets u - v$;\\
                    \For{$k \in \llbracket 1, K \rrbracket$}{ 
                        Decrease $P_{i, j_k}^{(k)} \gets P_{i, j_k}^{(k)} - v$;\\
                        \If{$P_{i, j_k}^{(k)} = 0$}{ Increment $j_k \gets j_k +
                        1$; }
                    }
                }
            } 
			\caption{North-West Corner Gluing.}
			\label{alg:NWC}
		\end{algorithm}
	\end{minipage}
\end{figure}

Noticing that \cref{eqn:G_discrete} only requires the $n_1\times \cdots
\times n_K$-tensor $\rho := \gamma_{1, \cdots, K}$, it is possible to only store
the indices $(j_1, \cdots, j_K)$ such that $\rho_{j_1, \cdots, j_K} > 0$, as
well as the corresponding weights $\rho_{j_1, \cdots, j_K}$. This avoids the
prohibitive memory cost of storing the full tensor $\gamma$, and takes advantage
of the sparsity of the multi-coupling $\gamma$: if each $\pi_k$ is an extremal
point of $\Pi(a, b_k)$, we conjecture that $\NWC(\pi_1, \cdots, \pi_K)$ is an
extremal point of $\Pi(a, b_1, \cdots, b_K)$, and thus $\#\supp \gamma \leq n +
\sum_k n_k - K$ (adapting techniques from \cite[Theorem
2]{anderes2016discrete}).

Thanks to \cref{eqn:G_discrete} we formalise the fixed-point iterations in the
discrete case in \cref{alg:fp}.
\begin{figure}[ht]
	\centering
	\begin{minipage}{.95\linewidth}
		\begin{algorithm}[H]
			\SetAlgoLined \KwData{barycentre coefficients $(\lambda_k) \in
			\Delta_K$, for $k \in \llbracket 1, K \rrbracket$, support of
			$\nu_k$: $Y_k \in \R^{n_k\times d_k}$, weights of $\nu_k$: $b_k \in
			\Delta_{n_k}$ and cost function $c_k:
			\R^d\times\R^{d_k}\longrightarrow\R_+$. Number of iterations $T$,
			initial size $n\geq 1$ and stopping criterion $\alpha \geq 0$.}
			\KwResult{Barycentre $\mu_T = \sum_{i=1}^{N_t}a_i^{(T)}
            \delta_{x_i^{(T)}}$.}
			\textbf{Initialisation:} Choose $\mu_0 =
			\sum_{i=1}^{n}a_i^{(0)}\delta_{x_i^{(0)}}$ with $a^{(0)} \in
			\Delta_n$ and $X^{(0)} \in \R^{n\times d}$.\\
			\For{$t \in \llbracket 0, T- 1 \rrbracket$}{ \For{$k \in \llbracket
                1, K \rrbracket$}{ Solve the OT problem: $\pi^{(k)}
                \in\underset{\pi \in \Pi(a^{(t)}, b_k)}{\argmin}\
                \sum_{i,j}\pi_{i,j}c_k(x_i^{(t)}, y_{k, j})$;}
                Compute $\gamma = \NWC(\pi^{(1)}, \cdots, \pi^{(K)})$;\\
                Compute $\rho = \gamma_{1, \cdots, K} = 
                \left[\sum_i \gamma_{j_1,
                \cdots, j_K}\right]_{j_1,\cdots, j_k}$ and write 
                $\supp\rho = \left((j_1^{(i)}, \cdots, j_K^{(i)})
                \right)_{i=1}^{N_t}$;\\
                \For{$i \in \llbracket 1, N_t \rrbracket$}{
                    Compute $x_i^{(t+1)} = B(y_{1, j_1^{(i)}}, \cdots, 
                    y_{K, j_K^{(i)}})$ and 
                    $a_i^{(t+1)} = \rho_{j_1, \cdots, j_K}$;\\
                } \If{$\W_2^2(\mu_{t+1}, \mu_t) < \frac{\alpha}{N_t} \|X^{(t)}\|_2^2$}{
				Declare convergence and terminate. } } \KwRet{$a^{(T)},
				X^{(T)}$}
			\caption{Discrete iteration of $G$.}
			\label{alg:fp}
		\end{algorithm}
	\end{minipage}
\end{figure}
Given our considerations on the support of NWC gluing, we expect (without
formal proof) the upper bound $\#\supp \mu_T \leq n + T(\sum_k n_k) - TK$. This
is the same conclusion as \cite{beier2025tangential}, which they also state
without proof about their Gromov-Wasserstein fixed-point iteration
\cite[Algorithm 5.2]{beier2025tangential}. From a memory perspective, the
algorithm does not require the storage of each $\gamma \in \R^{N_t\times
n_1\times\cdots\times n_K}$, as remarked for the NWC algorithm.

In some specific cases, the expression in \cref{eqn:G_discrete} becomes simpler.
If the weights $a$ and $b_k$ are all uniform and $n=n_1=\cdots=n_K$, then the
Birkhoff-von-Neumann Theorem allows the choice of each transport plan $\pi_k$ as
permutation assignments $[\pi_k]_{i,j} =
\frac{1}{n}\mathbbold{1}(\sigma_k(i)=j)$. In this case, the expression of
$G(\mu)$ becomes:
\begin{equation}\label{eqn:G_discrete_unif}
    G(\mu) = \frac{1}{n}\sum_{i=1}^n\delta\left(B(y_{1, \sigma_1(i)}, \cdots, y_{K, \sigma_K(i)})\right).
\end{equation}
If one takes the barycentric projections of the OT plans $\pi^{(k)}$ in
\cref{eqn:G_discrete}, one obtains a discrete expression of $H$ (from
\cref{eqn:def_H}) written in \cref{eqn:H_discrete} and visualised in \cref{fig:H_discrete}.
\begin{equation}\label{eqn:H_discrete}
    H(\mu) = \left\{\sum_{i=1}^na_i\delta\left[B\left((1/a_i)\sum_{j=1}^{n_1}\pi_{i,j}^{(1)}y_{1, j}, \cdots, (1/a_i)\sum_{j=1}^{n_K}\pi_{i,j}^{(K)}y_{K, j}\right)\right],\; \pi^{(k)} \in \Pi_{c_k}^*(\mu, \nu_k)\right\}.
\end{equation}
\begin{figure}[ht]
    \begin{center}
        \includegraphics[width=0.6\linewidth]{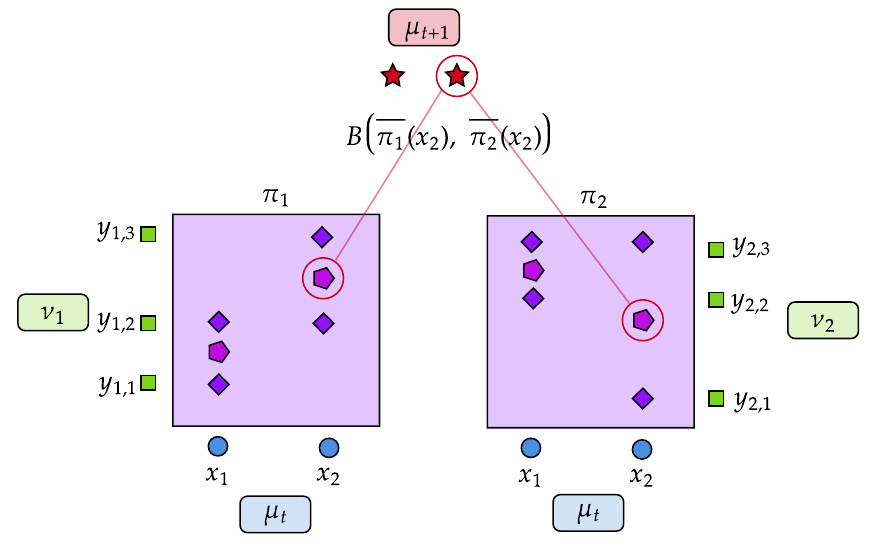}
    \end{center}
    \caption{Visual explanation of the discrete fixed-point iteration $H$. For
    each point $x_i$ in the support of $\mu_t$ and $k \in \llbracket 1, K
    \rrbracket$, we look at all the points $(y_{k, j_k})$ which are assigned
    from $x_i$ by the OT plan $\pi^{(k)}$ between $\mu$ and $\nu_k$, then take
    their barycenter $\oll{\pi}_k(x_i)$ (pentagons on the figure). The point
    $x_i$ in the support of $\mu_t$ is then sent to the point
    $B(\oll{\pi}_1(x_i), \cdots, \oll{\pi}_K(x_i))$ in $\mu_{t+1}$.}
    \label{fig:H_discrete}
\end{figure}
Contrary to $G$, for $H$ the number of points in the support of $\mu_t$ remains
the same, and the weights $a$ remain fixed. In this setting, the optimisation is
done solely on the positions, which can be seen as a Lagrangian formulation.
Note that in the squared-Euclidean case, \cref{eqn:H_discrete} is the formula
proposed in \cite[Equation 8]{cuturi14fast} and
\href{https://pythonot.github.io/gen_modules/ot.lp.html#ot.lp.free_support_barycentre}{currently
implemented} in the Python OT library \cite{flamary2021pot}. A technical
difference is that \cite{cuturi14fast} also proposes an optimisation over the
barycentre weights (by sub-gradient descent), while the fixed-point approach by
\cite{alvarez2016fixed} and ours do not. Furthermore, \cite{cuturi14fast}
suggests a computational simplification by using barycentric projections of
\textit{entropic} plans,
for which, as for $H$, there are no theoretical guarantees (to our knowledge).

The practical advantage of the map-supported expressions in
\cref{eqn:G_discrete_unif,eqn:H_discrete} over \cref{eqn:G_discrete} is
computational: since the support size of $\mu_t$ cannot increase, the
cost of computing the OT plans at Line 4 is smaller.
We shall see in \cref{sec:uniqueness_discrete} that in some
cases, Kantorovich solutions are almost-surely permutations for random supports.
While convenient, this expression only holds when all the measures have the same
amount of points, in contrast to the barycentric expression
\cref{eqn:H_discrete}.

\begin{figure}[ht]
	\centering
	\begin{minipage}{.95\linewidth}
		\begin{algorithm}[H]
			\SetAlgoLined \KwData{barycentre coefficients $(\lambda_k) \in
			\Delta_K$, for $k \in \llbracket 1, K \rrbracket$, support of
			$\nu_k$: $Y_k \in \R^{n_k\times d_k}$, weights of $\nu_k$: $b_k \in
			\Delta_{n_k}$ and cost function $c_k:
			\R^d\times\R^{d_k}\longrightarrow\R_+$. Number of iterations $T$,
			barycentre size $n\geq 1$, weights $a\in \Delta_n$ and stopping
			criterion $\alpha \geq 0$.} \KwResult{Barycentre $\mu_T =
			\sum_{i=1}^na_i\delta_{x_i^{(T)}}$.} \textbf{Initialisation:} Choose
			$\mu_0 = \sum_{i=1}^{n}a_i\delta_{x_i^{(0)}}$ with $X^{(0)} \in
			\R^{n\times d}$.\\
			\For{$t \in \llbracket 0, T- 1 \rrbracket$}{ \For{$k \in \llbracket
                1, K \rrbracket$}{ Solve the OT problem: $\pi^{(k)}
                \in\underset{\pi \in \Pi(a, b_k)}{\argmin}\
                \sum_{i,j}\pi_{i,j}c_k(x_i^{(t)}, y_{k, j})$; } \For{$i \in
                \llbracket 1, n \rrbracket$}{ Compute $x_i^{(t+1)} =
                B\left((1/a_i)\sum_{j=1}^{n_1}\pi_{i,j}^{(1)}y_{1, j}, \cdots,
                (1/a_i)\sum_{j=1}^{n_K}\pi_{i,j}^{(K)}y_{K, j}\right)$; }
                \If{$\W_2^2(\mu_{t+1}, \mu_t) < \frac{\alpha}{n} \|X^{(t)}\|_2^2$}{
                Declare convergence and terminate. } } \KwRet{$X^{(T)}$}
			\caption{Discrete iteration of $H$.}
			\label{alg:fp_H}
		\end{algorithm}
	\end{minipage}
\end{figure}

We present the iteration of $H$ as a cost-effective alternative to $G$,
which is in some sense a simplification of the Block-Coordinate Descent (BCD)
method, wherein the update with respect to the support $X \in \R^{n\times d}$
with transport plans $(\pi_k)$ fixed is done by computing:
\begin{equation}\label{eqn:BCD_update}
    X^* \in\underset{X \in \R^{n\times d}}{\argmin}\ 
    \sum_{k=1}^K \sum_{i=1}^n \sum_{j=1}^{n_k} \pi_{i,j}^{(k)}
    c_k(x_i, y_{k, j}).
\end{equation}
In practice, apart from the case of the squared Euclidean cost, the optimisation
in \cref{eqn:BCD_update} is not tractable, and one must resort to Gradient
Descent (GD) methods. BCD methods with GD for the update of $X$ can be seen as a
variant of the full GD method which minimises $X \longmapsto V(\frac{1}{n}\sum_i
\delta_{x_i})$, and we leave their study for future work.

\subsection{Correspondence of Gradient Descent with Fixed-Point Iterations}

The fixed-point method of \cite{alvarez2016fixed} applied to Bures-Wasserstein
barycentres also corresponds to a gradient descent algorithm with a specific
step size, as remarked by \cite{altschuler2021averaging}. This also holds for
discrete measures. Indeed, writing $X = \{x_1,\cdots, x_n\}$ and assuming $\mu_X
= \frac 1 n \sum_{i=1}^n \delta_{x_i}$, an alternative to fixed-point iterations
would be to apply a gradient descent directly on the non convex functional
$F:X\longmapsto \sum_{k=1}^K\lambda_k\T_{c_k}(\mu_X, \nu_k)$. For differentiable
costs $c_k$, assuming that $\nu_k = \frac 1 n \sum_{i=1}^n \delta_{y_{k, i}}$,
one step of such a gradient descent writes 
\begin{equation}
    \forall i \in \llbracket 1, n \rrbracket,\; x_i^{(t+1)} = x_i^{(t)} - \alpha  \sum_{k=1}^K \lambda_k \nabla_x c_k(x_i^{(t)},y_{k, \sigma_k^{(t)}(i)}),
    \label{eqn:grad_desc_F}
\end{equation} 
where we choose an element of $\Pi_{c_k}^*(\mu_{X^{(t)}}, \nu_k)$ induced by a
permutation $\sigma_k^{(t)}$ between $\{x_1^{(t)},\cdots, x_n^{(t)}\}$ and
$\{y_{k, 1},\cdots, y_{k, n}\}$. The whole optimisation algorithm consists in
alternating such gradient steps on $X$ with updates of the optimal assignments
$\{\sigma_k^{(t)}\}$, depending on the new point positions. In the fixed-point
approach, this gradient step on each $x_i^{(t)}$ is replaced by the computation
of $B(y_{1, \sigma_1^{(t)}(i)}, \cdots, y_{K,\sigma_K^{(t)}(i)})$, which
corresponds to a full descent on $X$ for a given configuration of assignments
before updating the said assignments (in other words, alternate minimisation).
For generic costs $c_k$, one may also use a gradient descent strategy to compute
barycentres $B(y_{1, \sigma_1^{(t)}(i)},\cdots,y_{K, \sigma_K^{(t)}(i)})$, that
is gradient descents on the $K$ functionals $ x \longmapsto
\sum_{k=1}^K{c_k}(x,y_{k, \sigma_k^{(t)}(i)})$, and such descents write exactly
as \cref{eqn:grad_desc_F}. In this case, the only difference between both
approaches is that the fixed point algorithm applies the whole descent on $X$
before updating assignments, while gradient descent on $F$ alternates steps of
gradient descent on $X$ with updates of the assignments. 

When $c_k = \|\cdot - \cdot\|_2^2$, both approaches are equivalent if the
gradient step is chosen as $\alpha = \frac 1 2$. Indeed, a gradient iteration on
$F$ writes
\[ \forall i \in \llbracket 1, n \rrbracket,\; x_i^{(t+1)} =
(1-2\alpha)x_i^{(t)} + 2\alpha \sum_{k=1}^K \lambda_k y_{k, \sigma_k^{(t)}(i)} =
\sum_{k=1}^K \lambda_k y_{k, \sigma_k^{(t)}(i)} .\] It follows that for $\alpha
= \frac 1 2$, one step of gradient descent computes directly the barycentre for
the current configuration of assignments $\{\sigma_k^{(t)}\}$, which is
precisely one iteration of the fixed-point algorithm. For different cost
functions, similar optimal steps may be formulated, but the step may depend on
$i$ and $x_i^{(t)}$.

Choosing the best strategy between the fixed point approach and the gradient
descent surely depends on the set of costs. When $B$ is easily computable (more
efficiently than by gradient descent), the fixed point algorithm moves the
points faster than gradient descent. However, it is not obvious what should be
the better option for complex costs $c_k$ in practice. More generally, one could
wonder if updating assignments more often (which is the case for the gradient
descent on $F$) might not help avoiding local minima of the whole functional
which is non convex in $X$. We did not observe this behaviour in practice in our
experiments and therefore recommend the fixed point approach as the default
choice.

\subsection{Discrete Uniqueness Discussion}\label{sec:uniqueness_discrete}

In this section, we investigate conditions to have uniqueness in the discrete
Kantorovich problem between measures $\mu = \sum_{i=1}^{n_x}a_i\delta_{x_i} \in
\mathcal{P}(\R^{d_x})$ and $\nu = \sum_{j=1}^{n_y}b_j\delta_{y_j} \in
\mathcal{P}(\R^{d_y})$:
\begin{equation}\label{eqn:discrete_kanto}
    \underset{\pi\in\Pi(a,b)}{\min}\ \sum_{i=1}^{n_x}\sum_{j=1}^{n_y} \pi_{i,j}c(x_i, y_j).
\end{equation}
For convenience, we introduce $X := (x_1, \cdots, x_{n_x}) \in \R^{n_x\times
d_x}$ and $Y := (y_1, \cdots, y_{n_y}) \in \R^{n_y\times d_y}$. The following
result shows that if the cost matrix $M := (X, Y) \longmapsto (c(x_i,
y_j))_{i,j} \in \R^{n_x\times n_y}$ is not orthogonal to a face of the
transportation polytope, then the discrete Kantorovich problem has a unique
solution. For convenience, we write $\pi \cdot M := \sum_{i,j}\pi_{i,j}M_{i,j}$.

\begin{prop}\label{prop:kanto_unique_C_not_pathological} Let $a \in
    \Delta_{n_x}$ and $b \in \Delta_{n_y}$ be fixed weights and $c: \R^{d_x}
    \times \R^{d_y} \longrightarrow \R_+$ a cost function. Consider the cost
    matrix function
    $$M := \app{\R^{n_x\times d_x}\times \R^{n_y \times d_y}}{\R^{n_x\times
    n_y}}{(X,Y)}{(c(x_i, y_j))_{i,j}},$$ and let $(X, Y) \in \R^{n_x\times
    d_x}\times \R^{n_y \times d_y}$. Denote by $\Extr \Pi(a,b)$ the (finite) set
    of extremal points of the transportation polytope $\Pi(a,b)$.
    \begin{equation}\label{eqn:condition_C_not_in_face_orthants}
        \underset{\pi\in\Pi(a,b)}{\min}\ \pi\cdot M(X, Y)\ \text{has\ a\ unique\ solution}\; \Longleftrightarrow \; M(X, Y) \not\in \underset{{\pi_1 \neq \pi_2 \in \Extr \Pi(a,b)}}{\bigcup} (\pi_1 - \pi_2)^\perp.
    \end{equation}
\end{prop}
\begin{proof}
    Since $\Pi(a,b)$ is convex and compact in $\R^{n_x\times n_y}$, by the
    Krein-Milman theorem, it is the convex hull of the set of its extreme
    points, denoted $\Extr \Pi(a,b)$. With the definition
    $$\Pi(a,b) = \left\{\pi \in \R^{n_x\times n_y} : \pi \geq 0,\; \pi
    \mathbf{1} = a,\; \pi^T \mathbf{1} = b\right\},$$ we see that $\Pi(a,b)$ is
    a polytope, and thus $\Extr \Pi(a,b)$ is finite. Since the Kantorovich
    problem is a linear problem, the set of optimal solutions is exactly the set
    of convex combinations of optimal extremal points. As a result, we have
    non-uniqueness in \cref{eqn:discrete_kanto} if and only if there exists
    $\pi_1 \neq \pi_2 \in \Extr \Pi(a,b) : \pi_1 \cdot M(X, Y) = \pi_2 \cdot
    M(X, Y)$. We conclude that uniqueness holds if and only if $\forall \pi_1
    \neq \pi_2 \in \Extr \Pi(a,b) : M(X,Y) \not\in (\pi_1 - \pi_2)^\perp.$
\end{proof}
A consequence of \cref{prop:kanto_unique_C_not_pathological} is that if
$M\#\Leb^{n_x\times d_x + n_y\times d_y}$ does not give mass to hyperplanes of
$\R^{n_x\times n_y}$, then the Kantorovich problem has a unique solution for
$\Leb^{n_x\times d_x + n_y\times d_y}$-almost-every $(X, Y)$. Furthermore, if
the measures have the same amount of points $(n_x= n_y)$ and the weights are
uniform, then the extreme points of $\Pi(a, b)$ are permutations, which provides
a theoretical justification for the convenient expression in
\cref{eqn:G_discrete_unif}.

\subsection{Application to Gaussian Mixture Model Barycentres}\label{sec:gmm}

In this section, we explain how our fixed-point algorithm can be applied to
compute barycentres between Gaussian Mixture Models (GMMs), providing a new
numerical method for the GMM barycentre notion introduced in \cite[Section
5]{delon2020wasserstein}. The notation $\PD$ will refer to the cone of positive
definite symmetric $d\times d$ matrices.

We consider the case where the measures are Gaussian Mixture Models, seen as
discrete measures over the space of Gaussian measures on $\R^d$: $\X :=
\mathcal{N} := \left\{\mathcal{N}(m, S) : m \in \R^d,\; S \in \PD\right\}$,
equipped with the 2-Wasserstein distance, which has a specific expression called
the \textit{Bures-Wasserstein distance}:
\begin{equation}\label{eqn:def_bures_distance}
    \W_2^2(\mathcal{N}(m_1, S_1), \mathcal{N}(m_2, S_2)) = \|m_1 - m_2\|_2^2 + \underbrace{\Tr\left(S_1 + S_2 - 2(S_1^{1/2}S_2S_1^{1/2})^{1/2}\right)}_{\bures(S_1, S_2):=}.
\end{equation}
Alternatively, one could see the same problem differently, setting $\X := \R^d
\times \PD$ equipped with the distance defined in \cref{eqn:def_bures_distance}.
To remind the definition of barycentres between Gaussian mixture models from
\cite{delon2020wasserstein}, we will consider measures that lie on the same
space of Gaussian measures: $\X = \Y_1 = \cdots = \Y_K = \mathcal{N}$. Next, we
choose cost functions $c_k$ on $\mathcal{N}$ as the squared Bures-Wasserstein
distance $\W_2^2$ scaled by $\lambda_k$. Given mixture models $\mu,\nu \in
\mathcal{P}(\mathcal{N})$ of the form
$$\mu = \sum_{i=1}^n a_i \delta_{\mathcal{N}(m_i, S_i)},\; \nu = \sum_{j=1}^m
b_j \delta_{\mathcal{N}(m_j', S_j')},$$ the Optimal Transport cost
$\T_{\W_2^2}(\mu, \nu)$ is the value of a discrete problem, which is precisely
the Mixed Wasserstein Distance introduced in \cite[Proposition
4]{delon2020wasserstein}:
\begin{equation}\label{eqn:def_mw2}
    \T_{\W_2^2}(\mu, \nu) = \underset{\pi \in \Pi(a, b)}{\min}\ \sum_{i,j}\pi_{i,j}\W_2^2(\mathcal{N}(m_i, S_i), \mathcal{N}(m_j', S_j')).
\end{equation}
Consider $K$ GMM measures $\nu_k$ written as:
$$\nu_k = \sum_{j=1}^{n_k}b_{k,j}\delta_{\mathcal{N}(m_{k,j}, S_{k,j})} \in
\mathcal{P}(\mathcal{N}),$$ their GMM barycentre cost with weights $(\lambda_k)$
for $\mu = \sum_{i=1}^na_i\delta_{\mathcal{N}(m_i, S_i)} \in
\mathcal{P}(\mathcal{N})$ reads:
\begin{equation}\label{eqn:mw2_bar}
    V(\mu) = \sum_{k=1}^K\lambda_k \underset{\pi_k \in \Pi(a, b_k)}{\min}\ \sum_{i,j} \pi_{i,j}\left(\|m_i - m_{k,j}\|_2^2 + \bures(S_i, S_{k,j})\right).
\end{equation}
We now turn to the expression of the ground barycentre function $B:
\mathcal{N}^K \longrightarrow \mathcal{N}$. This corresponds to a 2-Wasserstein
barycentre problem in the Gaussian case, which was first studied by
\cite{agueh2011barycenter} (showing existence and uniqueness in \cite[Theorem
6.1]{agueh2011barycenter}):
$$B(\mathcal{N}(m_1, S_1), \cdots, \mathcal{N}(m_K, S_K)) = \mathcal{N}(\oll{m},
\oll{S}),\; \oll{m} := \sum_{k=1}^K \lambda_k m_k,\; \oll{S} := \underset{S \in
\PD}{\argmin}\sum_{k=1}^K\lambda_k \bures(S, S_k).$$ A fixed-point formulation
of this problem is presented in \cite{alvarez2016fixed} as a particular case of
their study of the fixed-point algorithm for the ground cost $\|\cdot -
\cdot\|_2^2$ and absolutely continuous measures. This problem is presented again
in \cite{Bhatia_Bures_metric}, were they prove additional convergence
guarantees. We recall from \cite{alvarez2016fixed, Bhatia_Bures_metric} the
fixed-point algorithm to compute the barycentre of $K$ Gaussians
$(\mathcal{N}(m_k, S_k))$ and weights $(\lambda_1, \cdots, \lambda_K)$, which
consists in iterating the function $G_\mathcal{N}: \PD \longrightarrow \PD$:
\begin{equation}\label{eqn:def_iterates_bar_cov}
    G_\mathcal{N}(S) = S^{-1/2}\left(\sum_{k=1}^K\lambda_k (S^{1/2}S_kS^{1/2})^{1/2}\right)^2S^{-1/2}.
\end{equation}
Now that we have defined the ground barycentre map $B$, we can apply our
fixed-point algorithm to compute a barycentre. Given a reference GMM with $n$
components $\mu = \sum_{i=1}^n a_i \delta_{\mathcal{N}(m_i, S_i)}$, for $k \in
\llbracket 1, K \rrbracket$, solve the discrete Kantorovich problem between
$\mu$ and $\nu_k$ (\cref{eqn:def_mw2}) and choose $\pi_k \in \Pi_{\W_2^2}^*(\mu,
\nu_k)$. The GMM of $G(\mu)$ associated to the choice of plans $\pi_k \in \Pi(a,
b_k)$ in the iteration scheme is the GMM $\oll{\mu}$ defined by:
$$\oll{\mu} = \sum_{j_1, \cdots,
j_K}\sum_{i=1}^n\cfrac{1}{a_i^{K-1}}\pi^{(1)}_{i, j_1}\times \cdots \times
\pi^{(K)}_{i, j_K}\delta[B(\mathcal{N}(m_{1, j_1}, S_{1, j_1}), \cdots,
\mathcal{N}(m_{K, j_K}, S_{K, j_K}))].$$ As we argued in
\cref{sec:discrete_expressions}, it is computationally wise to consider a
variant of the fixed-point iterations which use the barycentric projections of
the couplings $\pi_k$ (see \cref{eqn:def_H}). To use this in the case of the
space $\mathcal{N}$, we need to choose a notion of convex combination in
$\mathcal{N}$ to be able to compute the images of the barycentric projections.
The most meaningful choice is a Wasserstein Gaussian barycentre, which
corresponds to using the ground barycentre map $B$ (this time with weights given
by the disintegration of the coupling in question).

\begin{remark}
    The metric space $(\mathcal{N}, \W_2)$ is not compact, however we consider
    discrete measures (GMMs). We will show how one can restrict $\mathcal{N}$ to
    a compact subset containing all barycentres. Combining \cite[Corollary
    3]{delon2020wasserstein} and \cite[Theorem 4.2, Equations 20 and
    21]{alvarez2016fixed}, shows that the barycentre is within a certain compact
    subset of $\mathcal{P}(\mathcal{N})$ of measures supported on Gaussians with
    covariances whose eigenvalues are in a segment $[r, R]$, where $0 < r < R$
    are explicit constants depending on the covariances of the components of
    $\nu_1, \cdots, \nu_K$. As for the means, they can be constrained to the
    convex hull of the means of the components of the mixtures $\nu_k$.
\end{remark}

\FloatBarrier
\newpage
\section{Numerical Illustrations}\label{sec:numerics}
In this section, we provide numerical experiments to illustrate the fixed-point
method (specifically its barycentric variant presented in \cref{alg:fp_H}) on
various toy datasets. All code from this section is available
\href{https://github.com/eloitanguy/ot\_bar}{in our companion Python toolkit}. A
numerical implementation of \cref{alg:fp}, which allows flexible support sizes,
is also possible, but computationally much less appealing than \cref{alg:fp_H}.

\subsection{Toy Example for Barycentre Computation}

We begin with a simple example of barycentre computation in $\R^2$ of two
discrete uniform measures with different support sizes and for the
square-Euclidean cost $c_k(x,y) = \|x-y\|_2^2$. We observe convergence to the
true barycentre in two iterations in \cref{fig:toy_barycentre_iterations}. The
support size increases from 10 to 19 at the first iteration and remains at 19 at
the final iteration.

\begin{figure}[H]
    \centering
    \includegraphics[width=\linewidth]{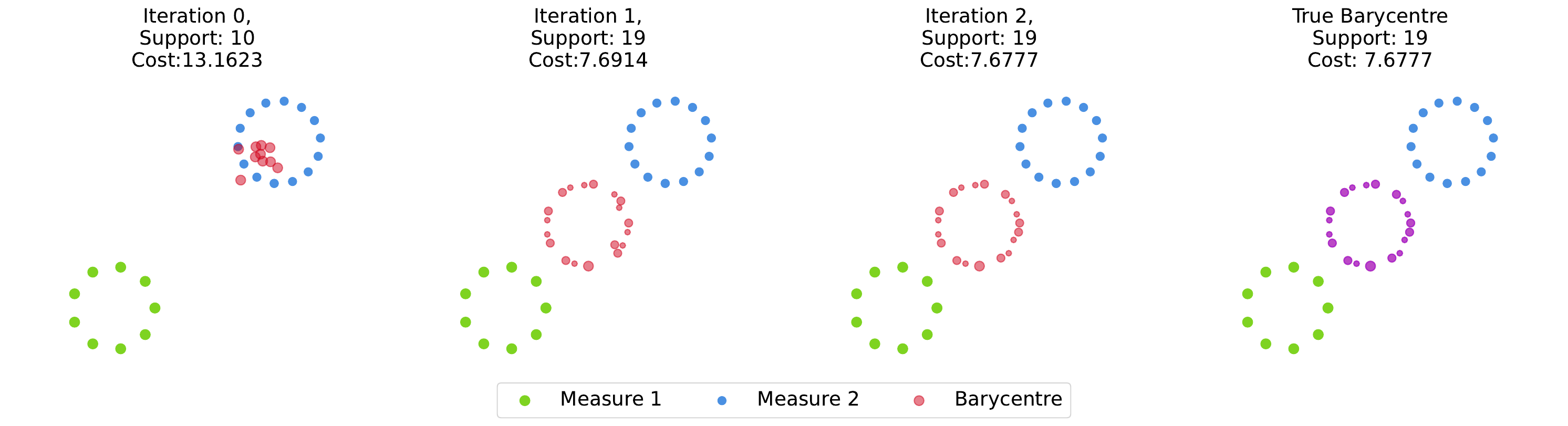}
    \caption{Iterations of \cref{alg:fp} for the square-Euclidean cost, and
    comparison with the true $\W_2^2$ barycentre.}
    \label{fig:toy_barycentre_iterations}
\end{figure}

\subsection{Illustration with Norm Powers}\label{sec:ex_bar_pq}

We consider discrete measures in $\R^2$ for costs $c_k(x,y) = \|x-y\|_p^q$, as
illustrated in \cref{fig:p_norm_q_barycentre}.

\begin{figure}[H]
    \centering
    \begin{adjustbox}{valign=c}
        \begin{subfigure}[b]{0.49\textwidth}
            \centering
            \includegraphics[width=\linewidth]{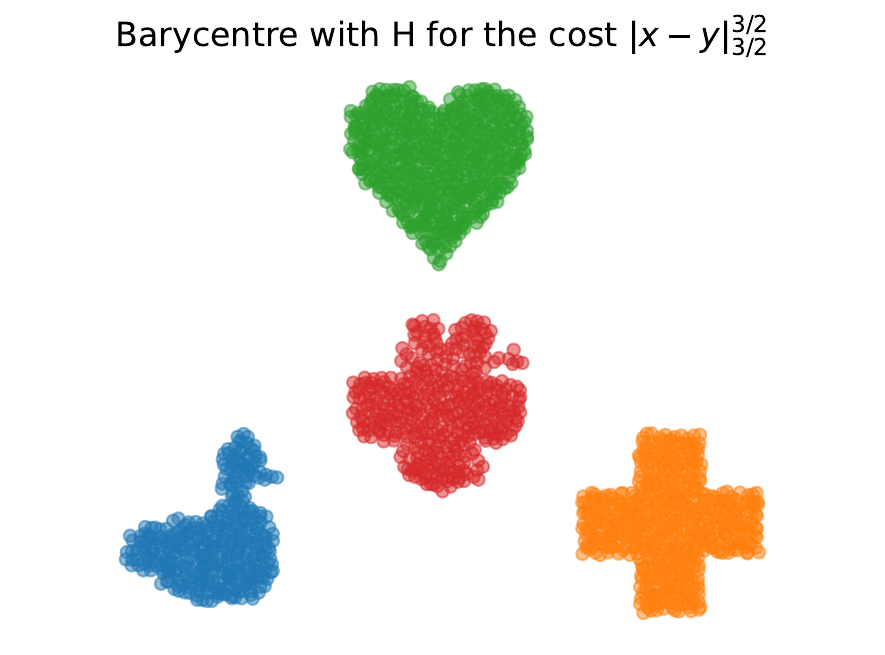}
        \end{subfigure}
    \end{adjustbox}
    \hfill
    \begin{adjustbox}{valign=c}
        \begin{subfigure}[b]{0.49\textwidth}
            \centering
            \includegraphics[width=\linewidth]{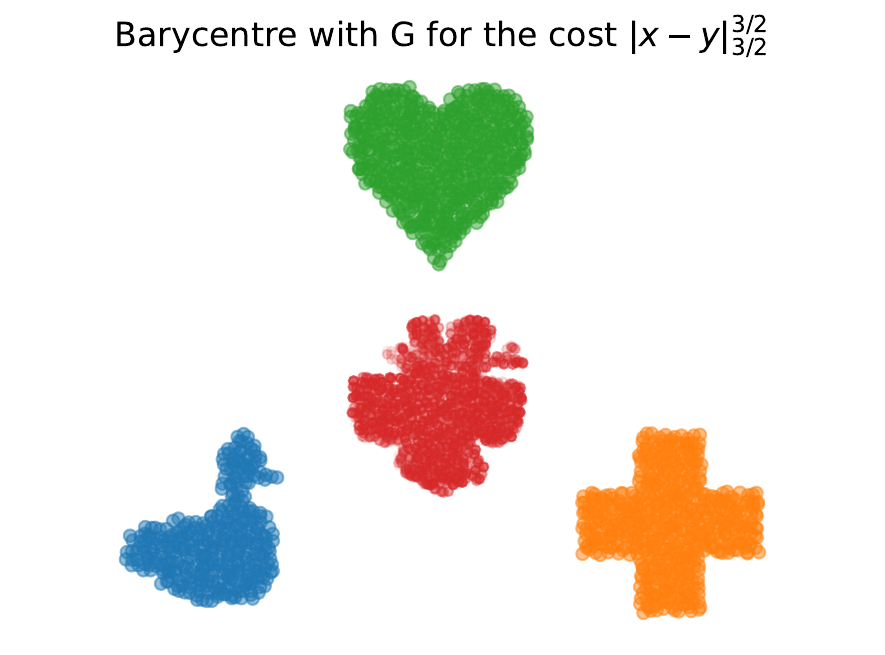}
        \end{subfigure}
    \end{adjustbox}
    \caption{Barycentres with initial support size $n=400$ for $(p, q) =
    (\frac{3}{2}, \frac{3}{2})$ of three measures with sizes 561, 382, 629.}
    \label{fig:p_norm_q_barycentre}
\end{figure}

In \cref{fig:p_norm_q_barycentre_iterations_G}, we observe that for $(p,
q) = (\frac{3}{2}, \frac{3}{2})$, the iterates of $G$ (\cref{alg:fp}) have an
energy that converges in one iteration, but the support size continues to grow
at iteration 2. As for $H$ (\cref{alg:fp_H}), we observe in
\cref{fig:p_norm_q_barycentre_iterations} convergence in one iteration. In
\cref{fig:p_norm_q_barycentres_grid}, we present barycentres for various pairs
$(p, q)$ using iterates of $H$.

\begin{figure}[H]
    \centering
    \begin{adjustbox}{valign=c}
        \begin{subfigure}[b]{0.5\textwidth}
            \centering
            \includegraphics[width=\textwidth]{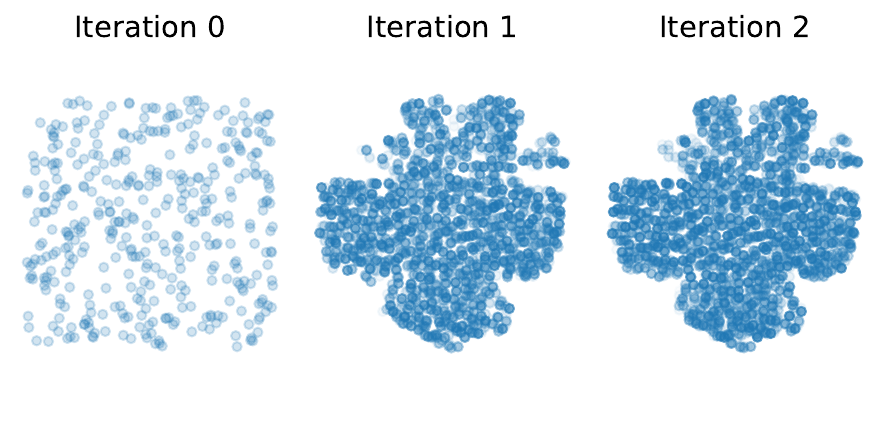}
            \caption{Fixed-point iterations for $(p, q) = (\frac{3}{2},
            \frac{3}{2})$.}
        \end{subfigure}
    \end{adjustbox}
    \hfill
    \begin{adjustbox}{valign=c}
        \begin{subfigure}[b]{0.2\textwidth}
            \centering
            \includegraphics[width=\textwidth]{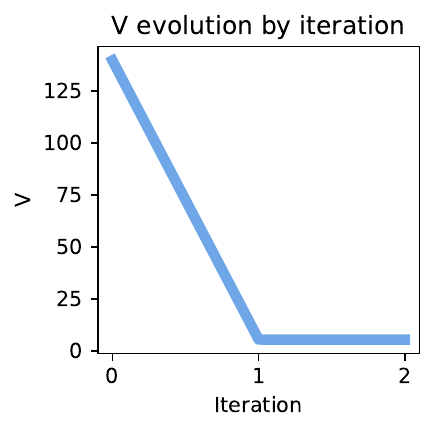}
            \caption{Energy $V$.}
        \end{subfigure}
    \end{adjustbox}
    \hfill
    \begin{adjustbox}{valign=c}
        \begin{subfigure}[b]{0.2\textwidth}
            \centering
            \includegraphics[width=\textwidth]{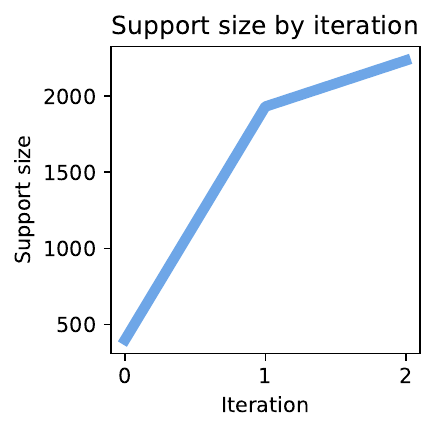}
            \caption{Support size.}
        \end{subfigure}
    \end{adjustbox}
    \caption{Convergence of the iterations of $G$ (\cref{alg:fp}).}
    \label{fig:p_norm_q_barycentre_iterations_G}
\end{figure}

\begin{figure}[H]
    \centering
    \begin{adjustbox}{valign=c}
        \begin{subfigure}[b]{0.6\textwidth}
            \centering
            \includegraphics[width=\textwidth]{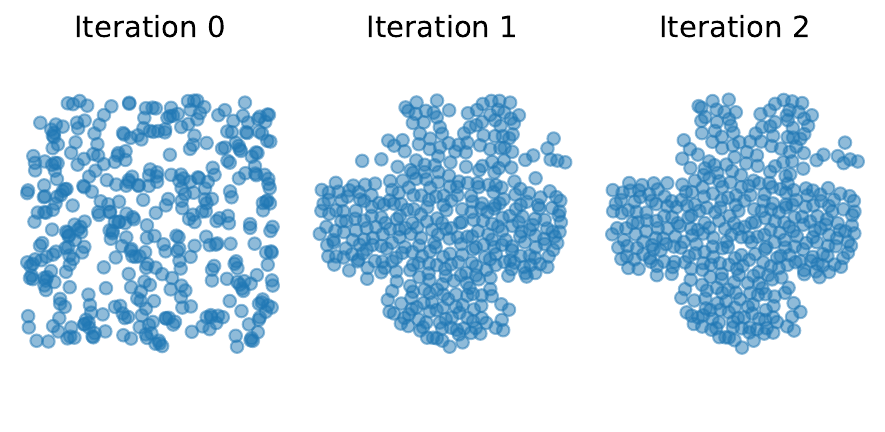}
            \caption{Fixed-point iterations for $(p, q) = (\frac{3}{2},
            \frac{3}{2})$.}
        \end{subfigure}
    \end{adjustbox}
    \hfill
    \begin{adjustbox}{valign=c}
        \begin{subfigure}[b]{0.3\textwidth}
            \centering
            \includegraphics[width=.6\textwidth]{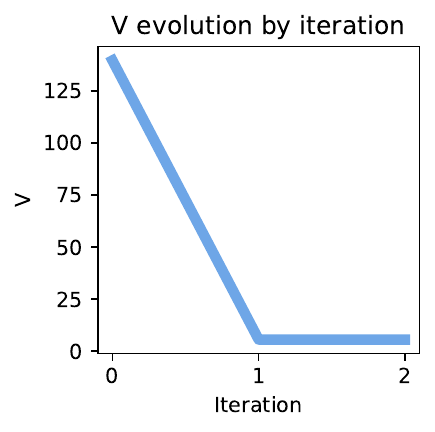}
            \caption{Barycentre energy $V$ of the iterations.}
        \end{subfigure}
    \end{adjustbox}
    \caption{Convergence of the iterations of $H$ (\cref{alg:fp_H}).}
    \label{fig:p_norm_q_barycentre_iterations}
\end{figure}

\begin{figure}[H]
    \centering
    \includegraphics[width=.6\linewidth]{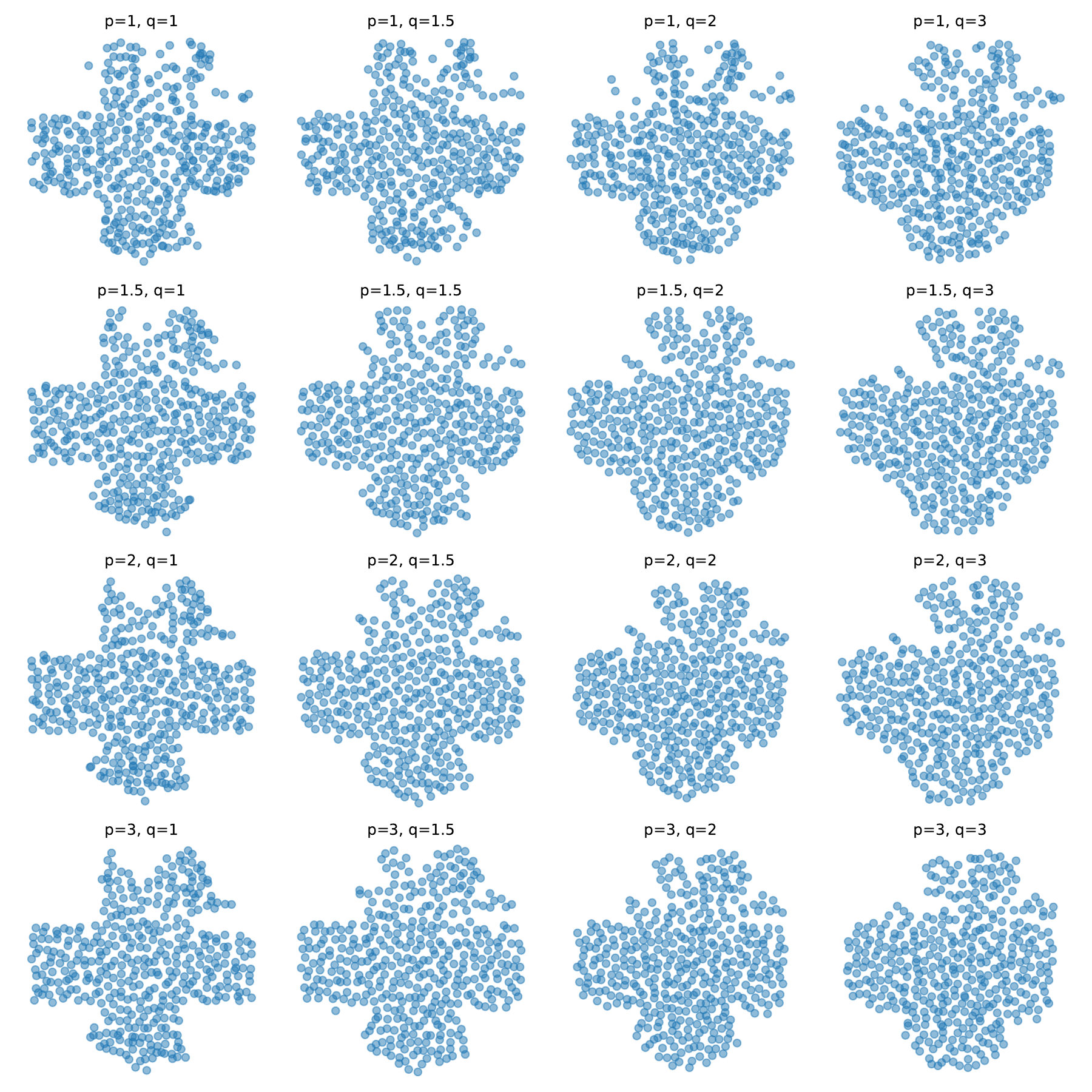}
    \caption{Barycentres for the cost $\|x-y\|_p^q$ for different values of $(p, q)$.}
    \label{fig:p_norm_q_barycentres_grid}
\end{figure}

In the following, we consider a different setting where two of the three
target measures are identical, and with a different third target. This will
allow us to study the robustness properties of the associated barycentre, seeing
the third different measure as an outlier. We represent the target measures and
a barycentre in \cref{fig:bary_point_cloud_2ducks_1heart}, and compare different
barycentres varying the parameters $(p,q)$ of the cost $\|\cdot - \cdot\|_p^q$
in \cref{fig:bary_grid_point_cloud_2ducks_1heart}. We observe that the
barycentre obtained for $q=1$ always takes the shape of the duck, as this power
allows for greater robustness to outliers (here the heart-shaped cloud),
regardless of the norm. The influence of the third point cloud becomes
increasingly evident as $p$ and $q$ grow.

\begin{figure}[H]
    \centering
    \includegraphics[width=0.4\linewidth]{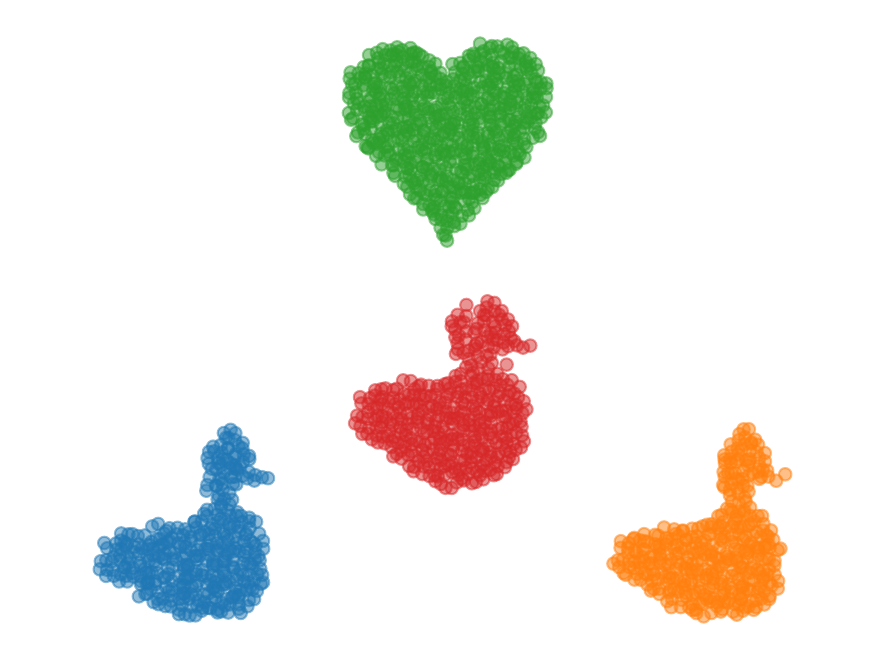}
    \caption{Barycentre of three point clouds for the cost $\|\cdot-\cdot\|_{3/2}^{3/2}$. }
    \label{fig:bary_point_cloud_2ducks_1heart}
\end{figure}

\begin{figure}[H]
    \centering
    \includegraphics[width=.6\linewidth]{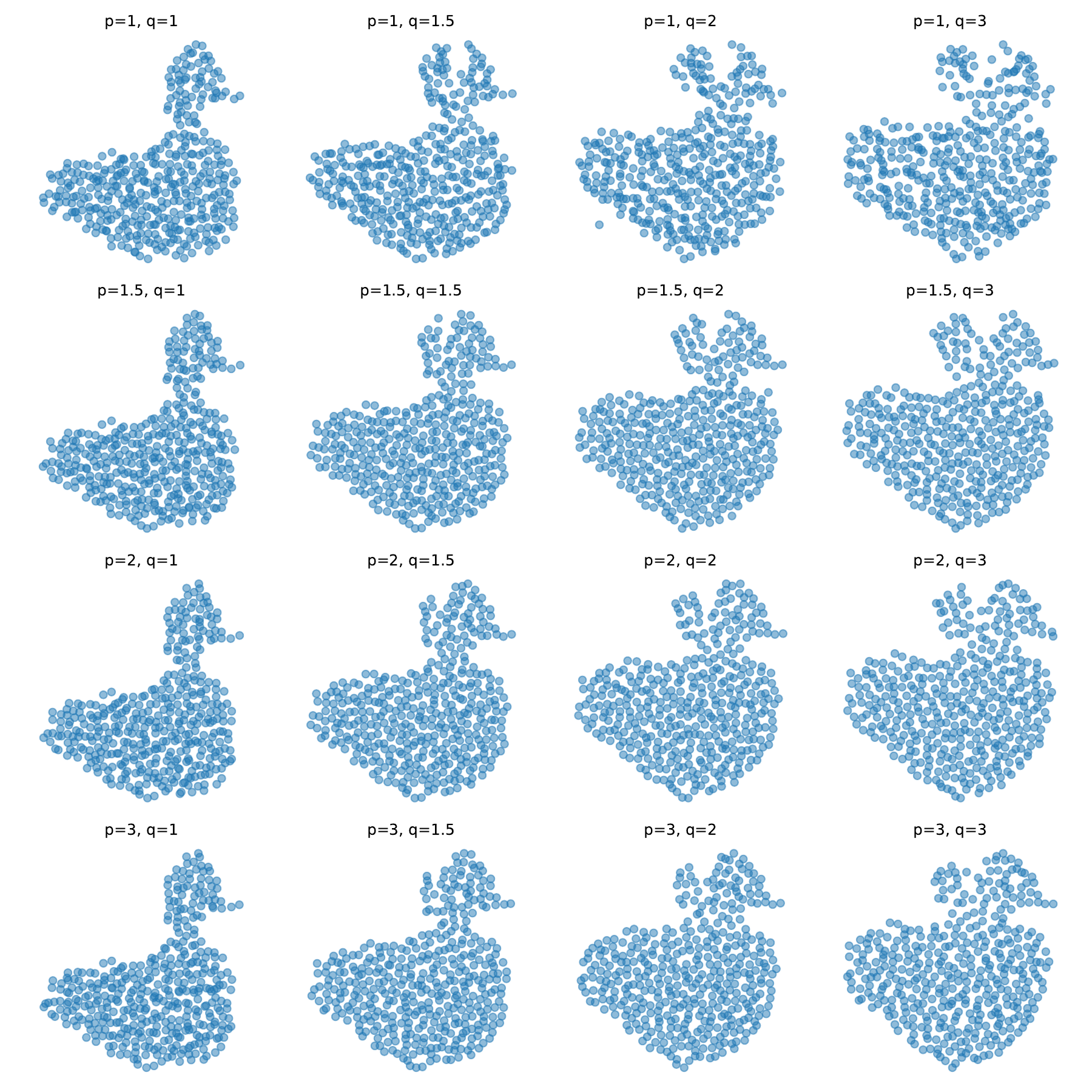}
    \caption{Barycentres for the measures of
        \cref{fig:bary_point_cloud_2ducks_1heart}, for norms $p$ raised to
        the power $q$.}
    \label{fig:bary_grid_point_cloud_2ducks_1heart}
\end{figure}

\FloatBarrier
\subsection{Study of the Support Size of Iterates of \texorpdfstring{$G$}{G}}

In this section, we study the support size $N$ of the final iteration of $G$
(\cref{alg:fp}). As discussed in \cref{sec:discrete_expressions}, we expect
(without formal proof) that the support size after $T$ iterations is
upper-bounded by $N_0 + T\sum_k n_k -TK$, where $N_0$ is the initial support
size and $n_k$ is the size of the $k$-th marginal. We verify this hypothesis on
numerical experiments on numerous configurations varying $N_0, (n_k), d, (d_k)$
with measure points and weights generated randomly and for the square-Euclidean
cost in \cref{fig:xp_G_support_size}. We observe that the upper-bound is indeed
respected, and that the algorithm attains convergence in a small number of
iterations.

\begin{figure}[H]
    \centering
    \begin{adjustbox}{valign=c}
        \begin{subfigure}[b]{0.74\textwidth}
            \centering
            \includegraphics[width=\textwidth]{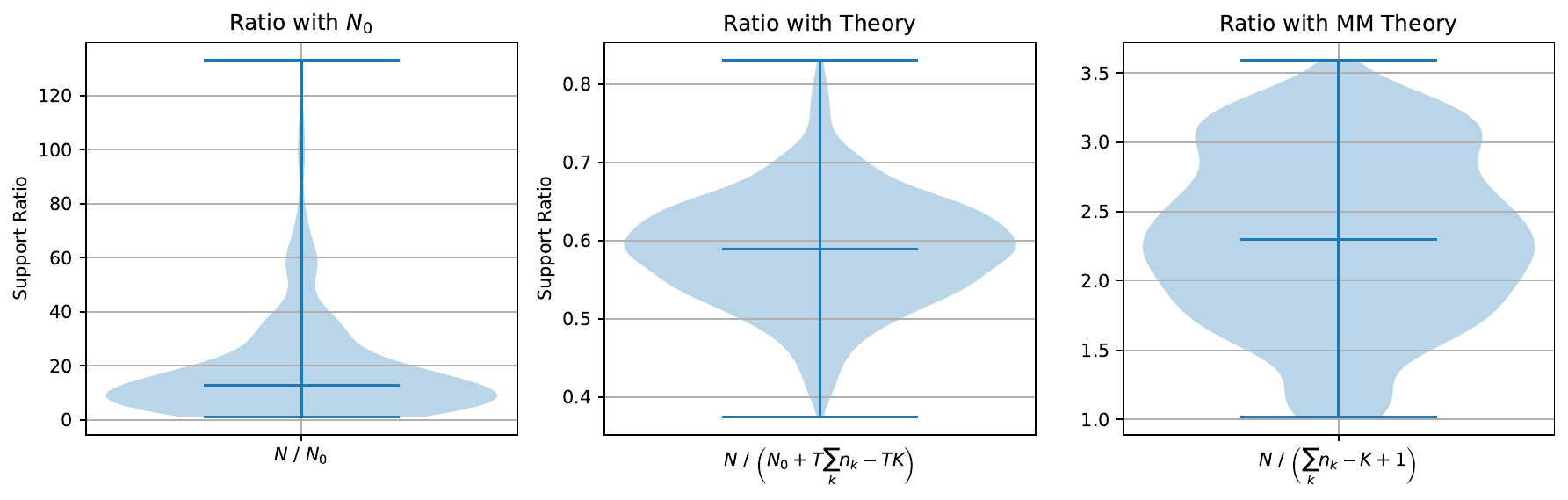}
            \caption{Ratios of the final support size $N$ with various support 
            sizes.}
        \end{subfigure}
    \end{adjustbox}
    \hfill
    \begin{adjustbox}{valign=c}
        \begin{subfigure}[b]{0.24\textwidth}
            \centering
            \includegraphics[width=\textwidth]{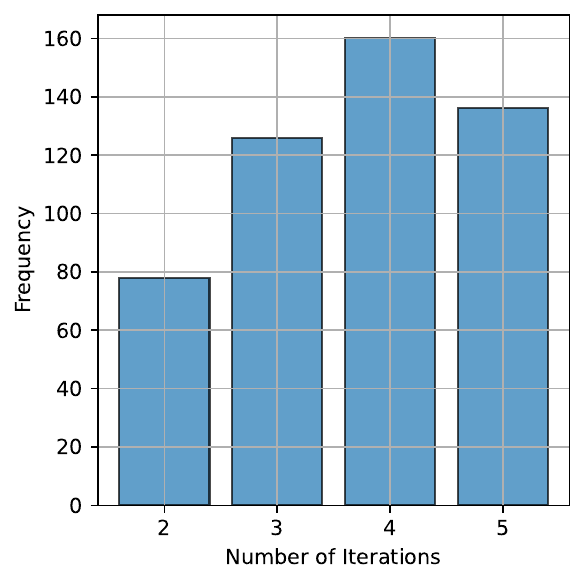}
            \caption{Iterations until convergence.}
        \end{subfigure}
    \end{adjustbox}
    \caption{Numerical study of the support size $N$ of iterates of $G$. We ran
    500 samples, each drawing a random number of measures $K \in \llbracket 2,
    10 \rrbracket$, random dimensions $d, d_1, \cdots, d_K \in \llbracket 1, 20
    \rrbracket$, random initial support sizes $N_0 \in \llbracket 10, 100
    \rrbracket$ and random measure sizes $n_k \in \llbracket 10, 100
    \rrbracket$.}
    \label{fig:xp_G_support_size}
\end{figure}

Running the same experiment as in \cref{fig:xp_G_support_size} with $N_0 = n_1 =
\cdots = n_K$ and uniform measure weights, we obtain, as expected in
\cref{eqn:G_discrete_unif} that the support size $N_t$ remains constant.

\subsection{Comparison with the Multi-Marginal Formulation}

Following \cref{eqn:multi_marginal_bar_equivalence}, the discrete OT barycentre
problem has a multi-marginal formulation, which can be written as follows, given
measures $\nu_k = \sum_{j=1}^{m_k} b_{k,j}\delta_{y_{k, j}}$:
\begin{equation}\label{eqn:discrete_mm}
    \underset{\pi \in \Pi(b_1, \cdots, b_K)}{\argmin}\ \sum_{j_1, \cdots, j_K} \pi_{j_1, \cdots, j_K}\sum_{k=1}^K c_k(B(y_{1, j_1}, \cdots, y_{K, j_K}), y_{k, j_k}).
\end{equation}
Numerical solvers for \cref{eqn:discrete_mm}, while slow, allow the computation
of the exact solution of the barycentre problem. Comparing this solution to the
output of our algorithm is technical, since the barycentric version of our
algorithm imposes the size of the support of the barycentre in addition to
imposing the weights, which introduces bias. We aim to illustrate that the speed
of the barycentric algorithm, with a quantitative study of the error with
respect to the multi-marginal ``ground truth''. Note that even in this
square-euclidean experiment, there is no widespread multi-marginal solver, which
is why we also contribute an implementation.

The experimental setup is the following: the $K$ measures $\nu_k$ are all
uniform measures with $n$ points in $\R^d$ drawn independently from
$\mathcal{N}(0, 1)$. For the fixed-point algorithm, the initial measure is also
taken as a uniform measure over $n$ points with $\mathcal{N}(0, 1)$ samples. We
compare different numbers of iterations of the fixed-point algorithm and
different choices of $n, d, K$. The plots show the ratios of the energy $V$ and
computation times for our algorithm divided by a Linear Programming
multi-marginal solver, plotting 30\% and 70\% quantiles across 10 samples for
each configuration. As expected in \cref{eqn:G_discrete_unif}, since in
this case the measures are uniform with a common support size, the iterates of
$H$ and $G$ are identical in this setting.

\begin{figure}[H]
    \centering
    \includegraphics[width=.9\linewidth]{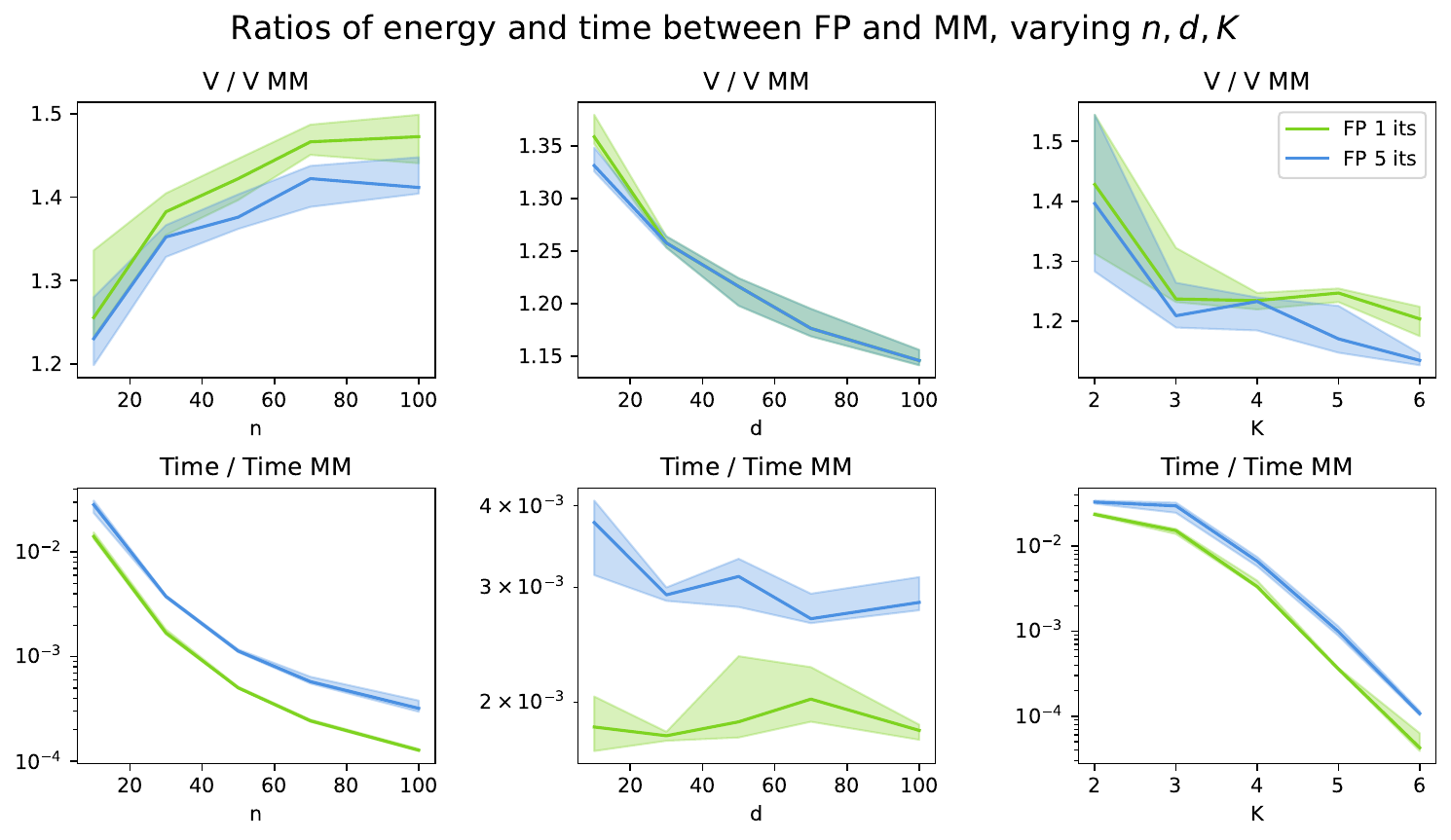}
    \caption{Comparing the fixed-point solver with a linear programming 
    multi-marginal solver. From left to right columns: varying $n$ with $d=10$ 
    and $K=3$; varying $d$ with $n=30$ and $K=3$; varying $K$ with $n=10$ and 
    $d=10$. The comparison is made by dividing the energy value $V$ (resp. 
    computation time) of the fixed-point solution by the multi-marginal 
    solution. The different curves correspond to $T = 1, 5, 10$ iterations 
    (legend in the top-right).}
    \label{fig:fp_vs_mm}
\end{figure}

From the results presented in \cref{fig:fp_vs_mm}, it appears that the
fixed-point algorithm converges in very few iterations, has an energy at most
50\% worse than the exact multi-marginal solution, and is orders of magnitude
faster, especially for larger measure sizes $n$ and for greater numbers of
marginals $K$. Note that for $n \geq 10$ and $K \geq 10$ for example, the
multi-marginal problem is computationally intractable.

To compare with similar barycentre support sizes, in \cref{fig:fp_vs_mm_N} we
experiment with fixed-point barycentres using $H$ (\cref{alg:fp_H}) with
$N_{\mathrm{FP}} = (n-1)K + 1$ points. The rationale behind this choice stems
from the fact that discrete measures with $n_1, \cdots, n_K$ points have a
barycentre with at most $\sum_k n_k - K + 1$ points (\cite[Theorem
2]{anderes2016discrete}\footnote{whose techniques are in fact not specific to
the cost $\|\cdot - \cdot\|_2^2$}).

\begin{figure}[H]
    \centering
    \includegraphics[width=.9\linewidth]{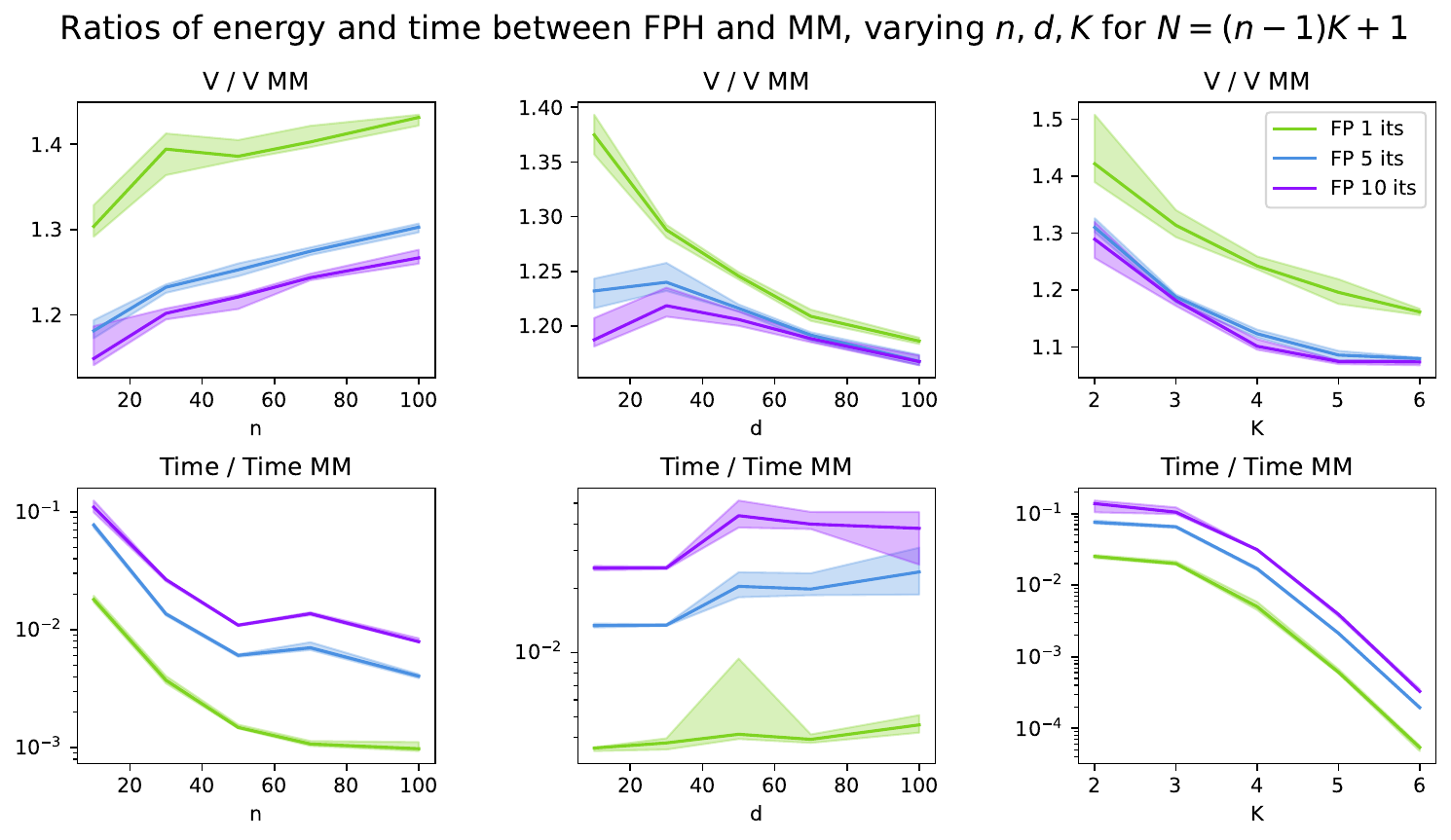}
    \caption{Comparing the fixed-point solver from \cref{alg:fp_H} for
    $N_{\mathrm{FP}} = (n-1)K + 1$ and the same setup as in
    \cref{fig:fp_vs_mm}.}
    \label{fig:fp_vs_mm_N}
\end{figure}

We now focus on the iterations of $G$ (\cref{alg:fp}) in the case of
uniform measures where the initialisation is taken with $n$ points and the
target measures have even spaced sizes $n_1 = \frac{n}{2} \cdots n_K = 2n$. This
ensures that iterates of $G$ differ from iterates of $H$, and we present the
results in \cref{fig:fpG_vs_mm}.

\begin{figure}[H]
    \centering
    \includegraphics[width=.9\linewidth]{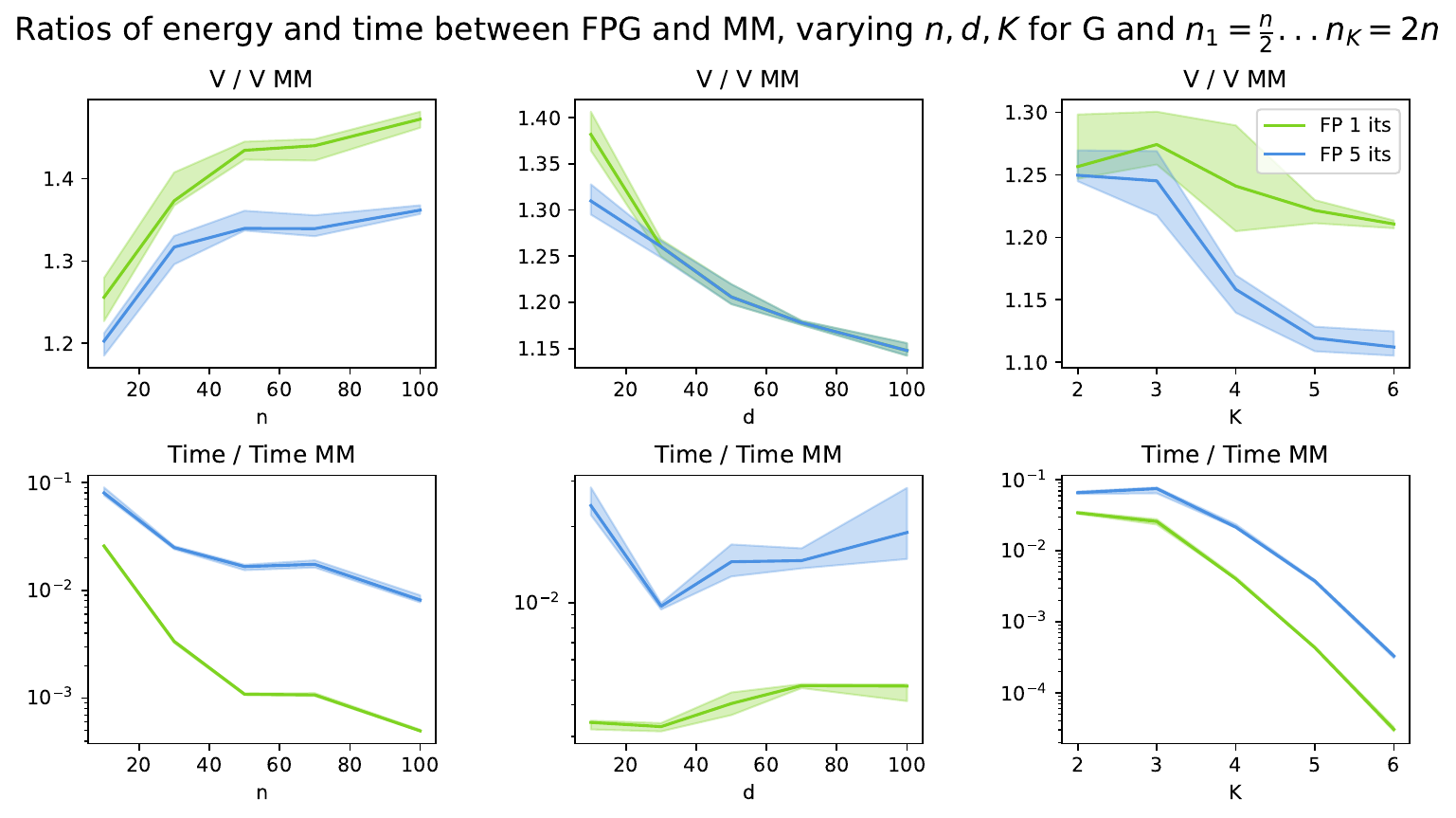}
    \caption{Comparing the fixed-point solver from \cref{alg:fp} with the MM
    solver, for an initialisation with $n$ points and target measures with
    different sizes $n_1 = \frac{n}{2} \cdots n_K = 2n$.}
    \label{fig:fpG_vs_mm}
\end{figure}

\cref{fig:fp_vs_mm,fig:fp_vs_mm_N,fig:fpG_vs_mm} suggests that the
fixed-point methods proposed in \cref{alg:fp,alg:fp_H}a re useful as a fast
approximate solvers for the barycentre problem, and that settings with larger
barycentre supports may require more iterations to converge. The main takeaway
is that our methods remain competitive for large supports and number of target
measures, yet its convergence speed and overall advantages are more pronounced
for smaller supports.

\subsection{Generalised Wasserstein Barycentre Computation}

In \cref{fig:gwb_W1_3D}, we illustrate the case where $c_k(x, y) = \|P_k x -
y\|_2$, where $P_k: \R^3 \longrightarrow \R^2$ is an orthogonal projection. The
problem finds a 3D measure whose projections attempt to match the reference 2D
measures, which we compare in \cref{fig:gwb_W1_proj}. This is a modification of
the exponent 2 from Generalised Wasserstein Barycentres
\cite{delon2021generalized}. 

\begin{figure}[H]
    \begin{adjustbox}{valign=c}
        \begin{subfigure}[b]{0.4\textwidth}
            \centering
            \includegraphics[width=\linewidth]{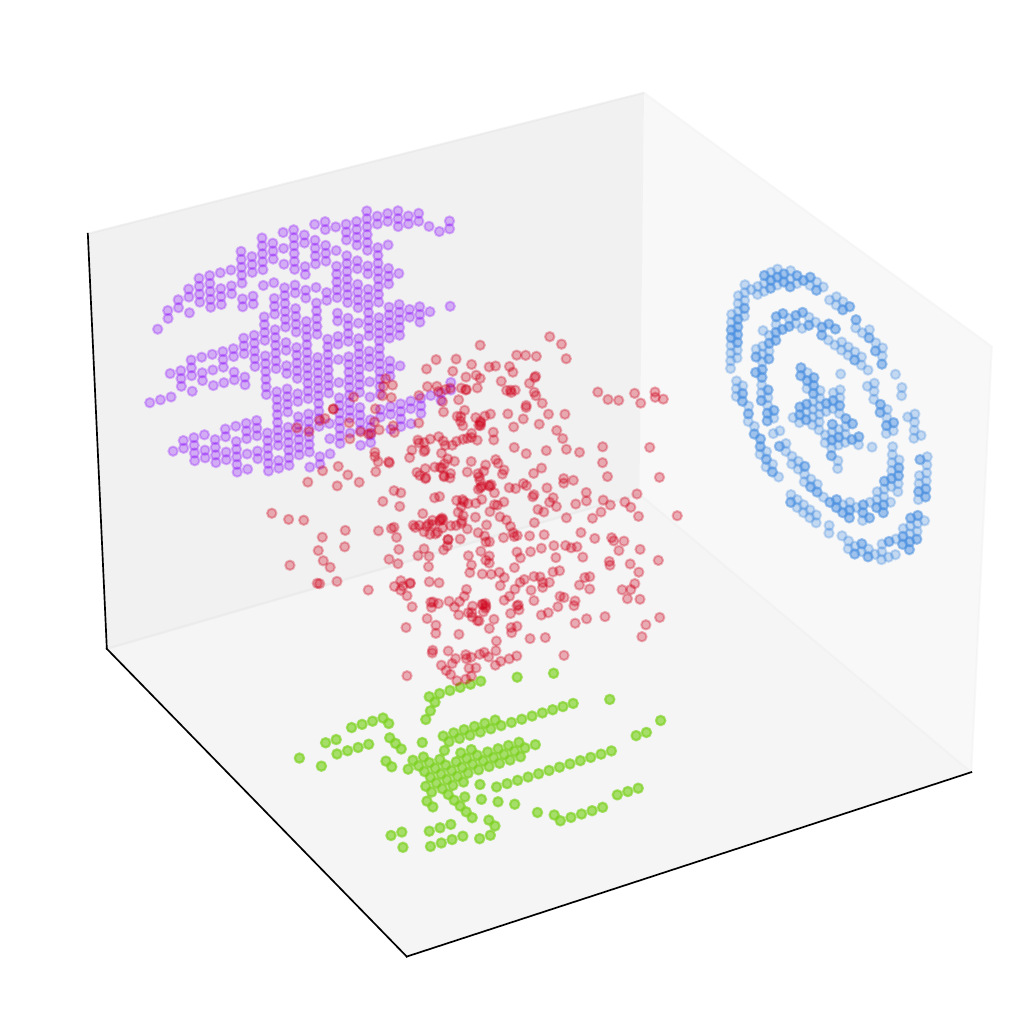}
            \caption{Barycentre in $\R^3$ with immersed measures $\nu_k \in
            \mathcal{P}(\R^2)$. }
            \label{fig:gwb_W1_3D}
        \end{subfigure}
    \end{adjustbox}
    \begin{adjustbox}{valign=c}
    \begin{subfigure}[b]{0.6\textwidth}
        \centering
        \includegraphics[width=\linewidth]{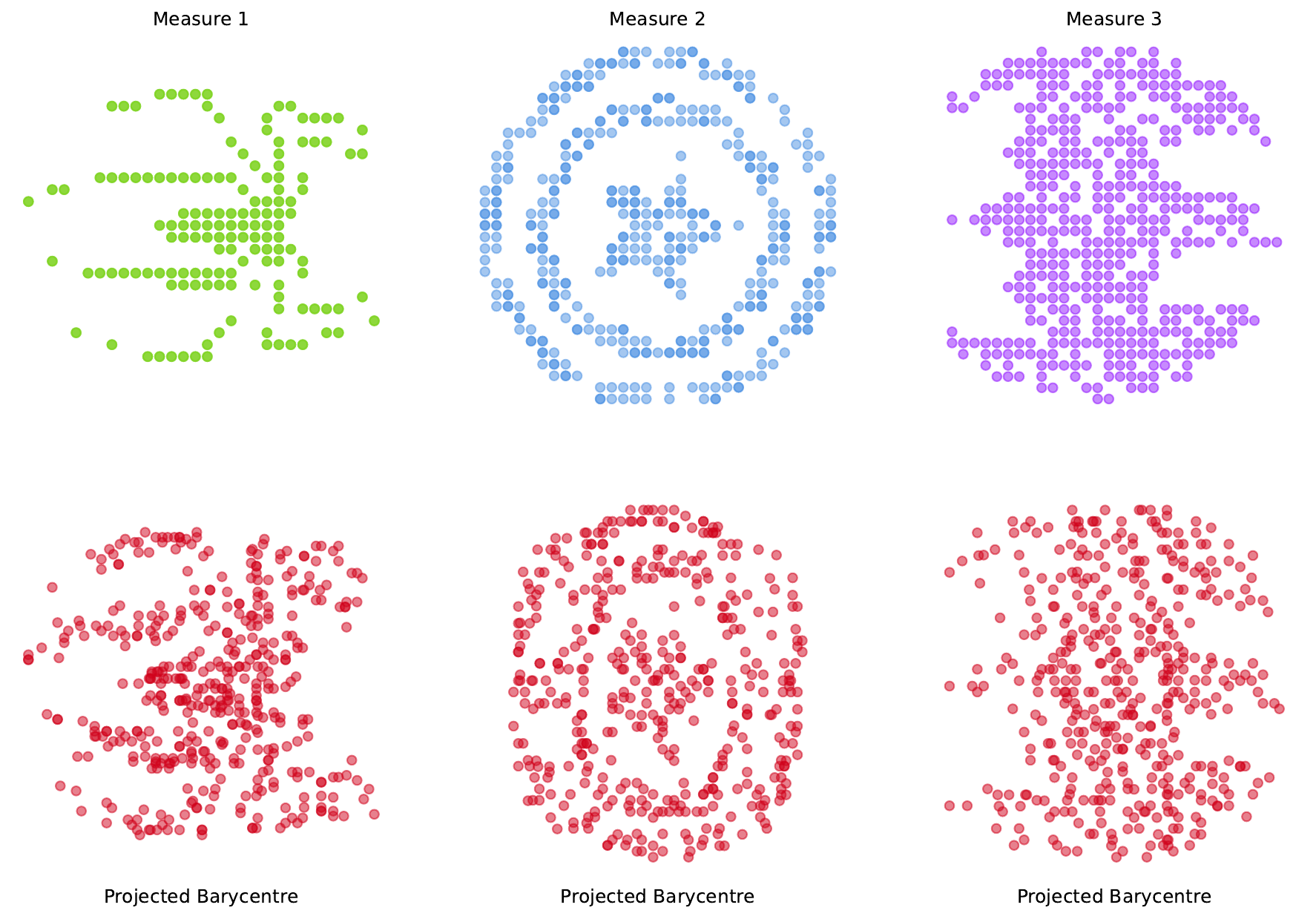}
    \caption{Projections of the barycentre into $\R^2$.}
        \label{fig:gwb_W1_proj}
    \end{subfigure}
    \end{adjustbox}
    \caption{Barycenter (using \cref{alg:fp_H}) with costs $c_k(x,y) = \|P_k
    x-y\|_2$, where $P_k$ are orthogonal projections from $\R^3$ to the three
    axes-aligned planes of the orthonormal basis. We provide an animation
    \href{https://github.com/eloitanguy/ot_bar/blob/main/examples/2d_proj_W1/barycenter_rotation.gif}{in
    the companion code}.}
\end{figure}

\subsection{Non-linear Generalised Wasserstein Barycentre Computation}

In this illustration, we look for a barycentre in $\R^2$ whose projections onto
different circles match measures on these circles. We choose the costs $c_k(x,
y) = \|P_k(x) - y\|_2^2$, where $P_k$ is the projection onto the circle $k$.
Since $P_k$ is not linear, this is a direct generalisation of
\cite{delon2021generalized}.

\begin{figure}[H]
    \centering
    \includegraphics[width=\linewidth]{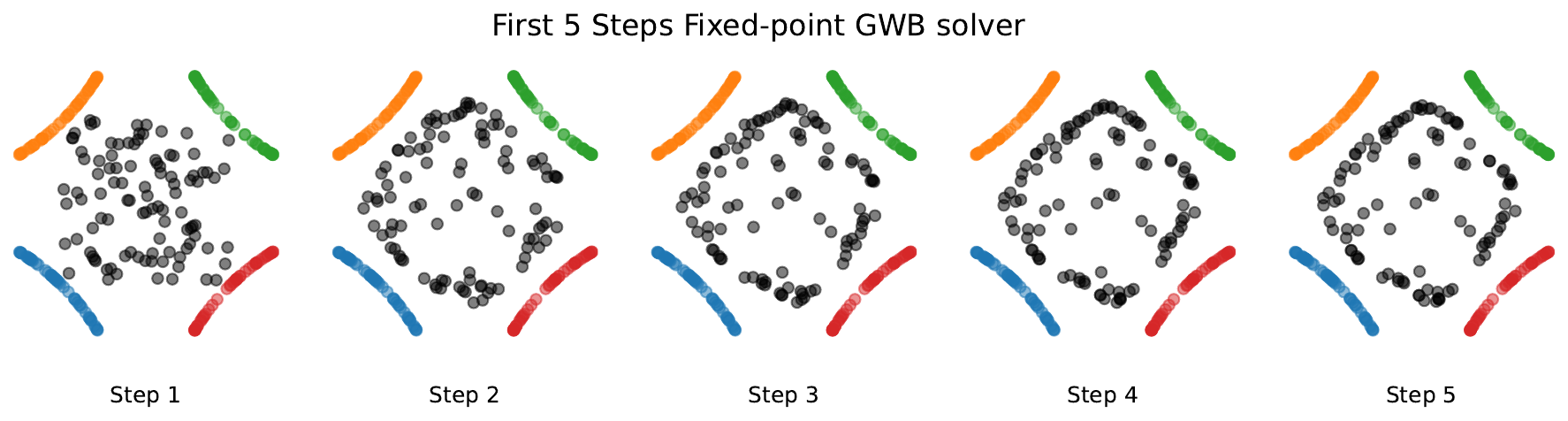}
    \caption{First 5 iterations of the fixed-point algorithm (\cref{alg:fp_H})
    for costs $c_k(x, y) = \|P_k(x) - y\|_2^2$, where $P_k$ are projections onto
    four different circles on which the $\nu_k$ are supported (plotted in
    colour). }
    \label{fig:circles_5_steps}
\end{figure}

In this instance, convergence happens quickly, but a stationary point is only
reached after about 5 iterations, as observed on the steps in
\cref{fig:circles_5_steps} and on the energy curve in \cref{fig:circles_V}.

\begin{figure}[H]
    \centering
    \includegraphics[width=.4\linewidth]{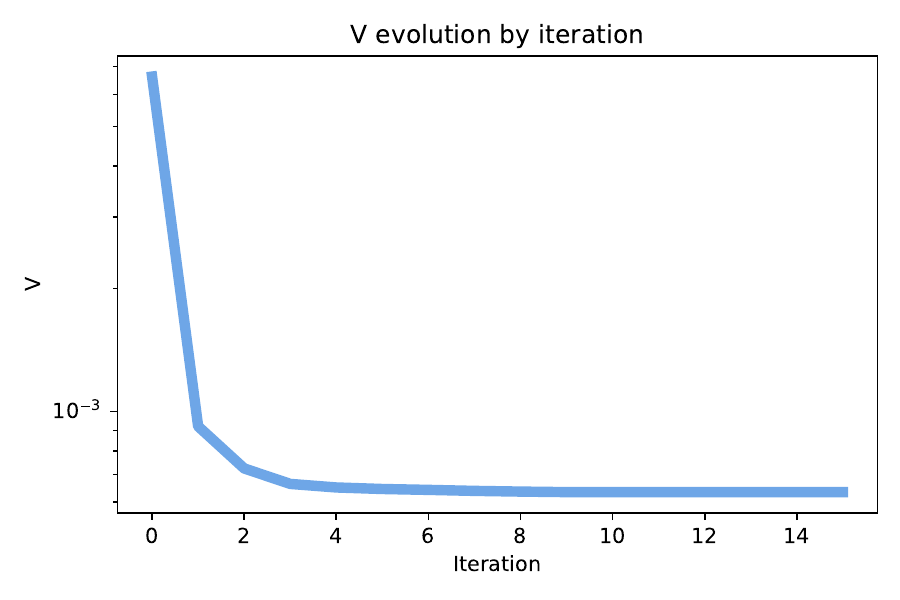}
    \caption{Barycentre energy $V$ of the fixed-point algorithm for $H$ across iterations.}
    \label{fig:circles_V}
\end{figure}

\subsection{Gaussian Mixture Model Barycentres}

We illustrate numerical solutions of the GMM Barycentre method introduced in
\cref{sec:gmm}. In \cref{fig:gmm_barycenters_comparison}, we compare the
multi-marginal solution with the output of our algorithm (we use
\cref{alg:fp_H}). 

\begin{figure}[H]
    \centering
    \includegraphics[width=.75\linewidth]{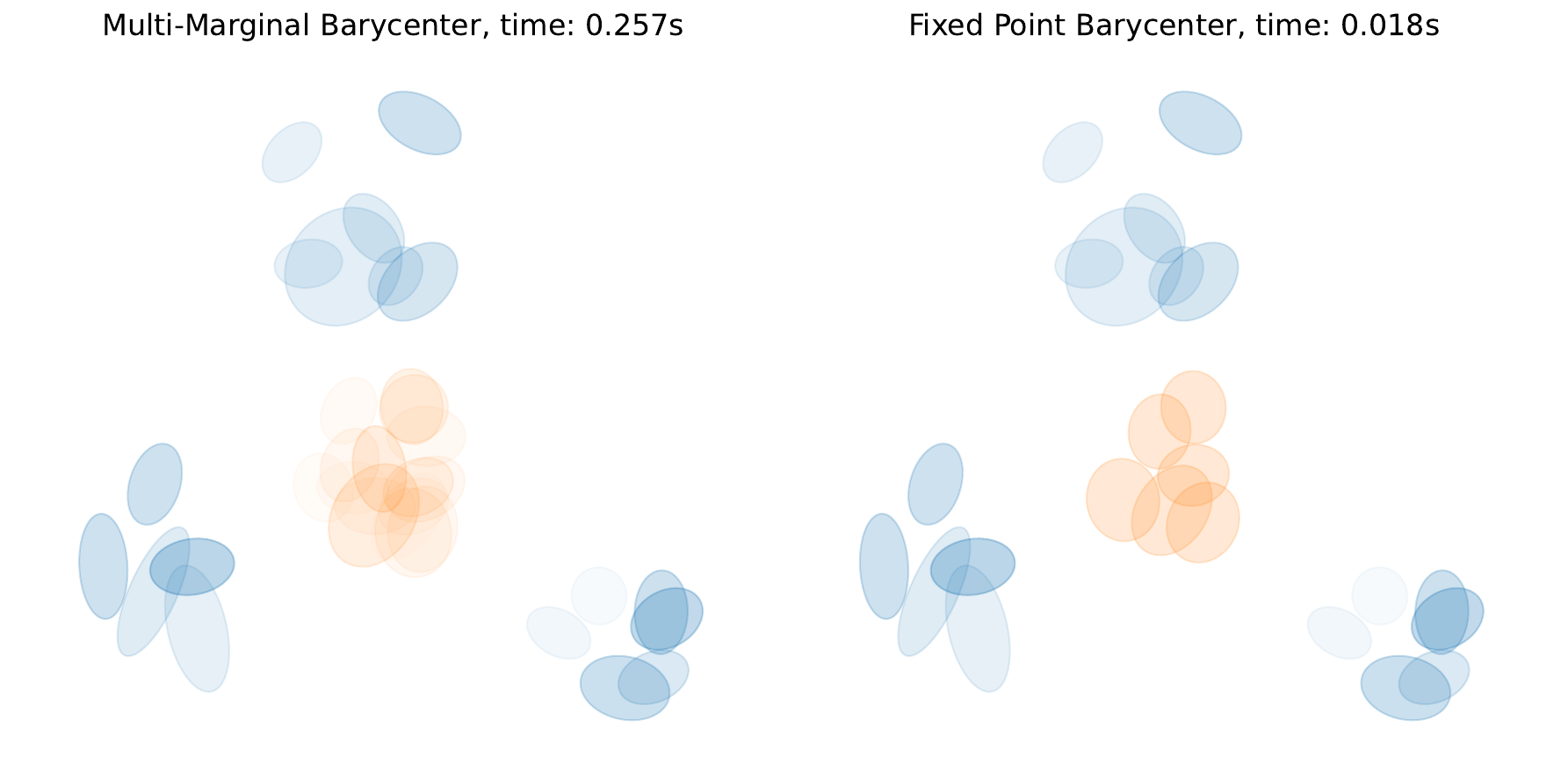}
    \caption{Left: multi-marginal solution for the GMM barycentre problem. Right: fixed-point solution for $n=6$ components.}
    \label{fig:gmm_barycenters_comparison}
\end{figure}

Finally, in \cref{fig:gmm_barycenter_interpolation} we illustrate barycentres
between 4 GMMs shown in \cref{fig:gmms} with different weights.

\begin{figure}[ht]
    \centering
    \includegraphics[width=.75\linewidth]{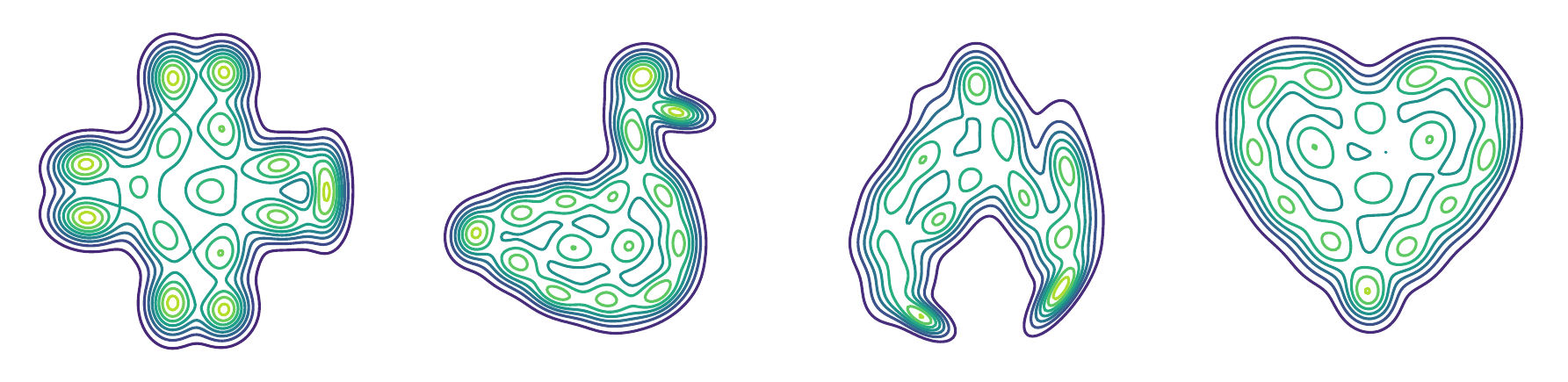}
    \caption{Four GMMs of which we will compute barycentres in
    \cref{fig:gmm_barycenter_interpolation}.}
    \label{fig:gmms}
\end{figure}

\begin{figure}[ht]
    \centering
    \includegraphics[width=.5\linewidth]{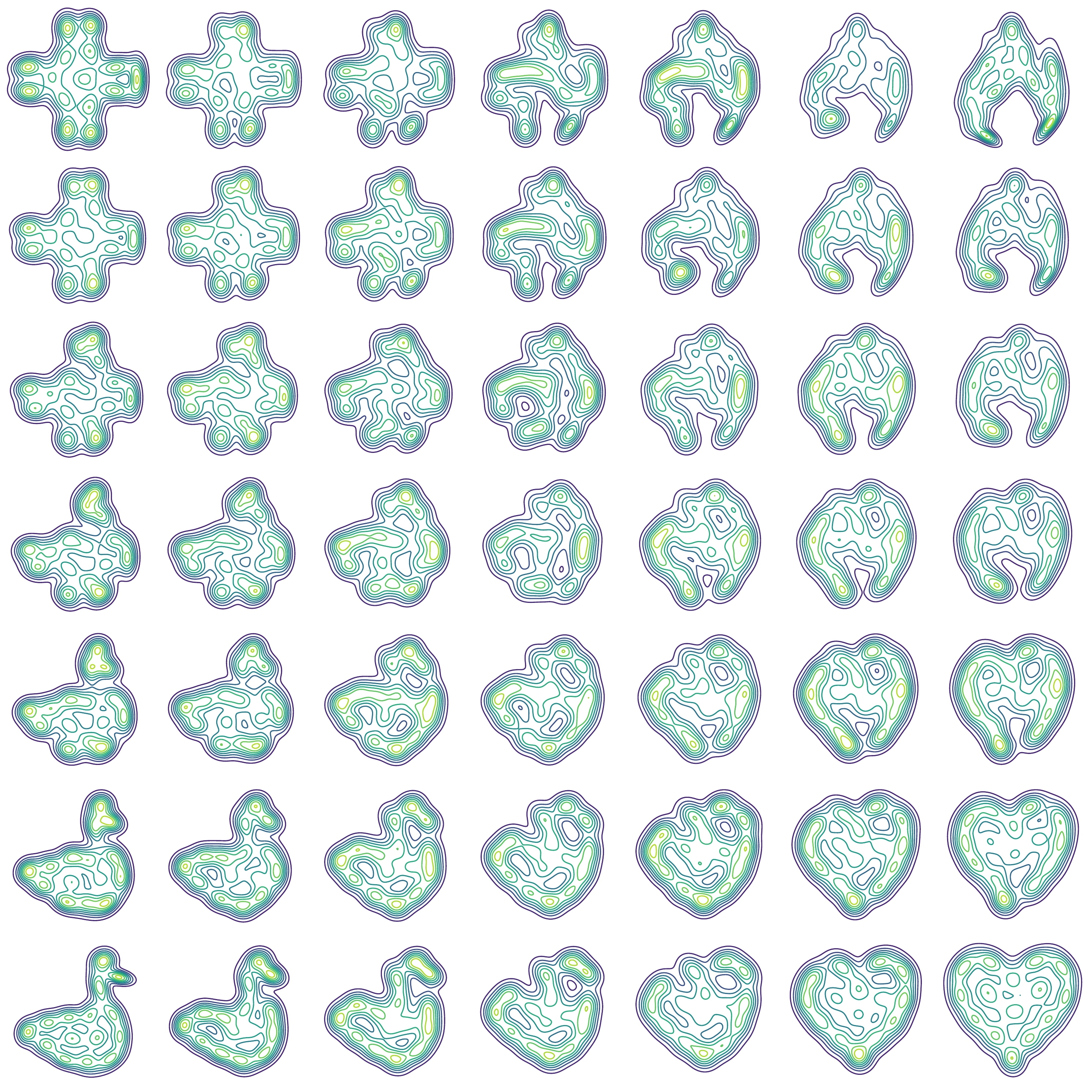}
    \caption{GMM barycentres between the four corner GMMs computed with the fixed-point solver with $n=15$ components. The GMMs are represented by the contours of their densities on $\R^2$.}
    \label{fig:gmm_barycenter_interpolation}
\end{figure}

\subsection{Colour Transfer on a Barycentre of Colour Distributions}\label{sec:ex_colour_transfer}

In this final experiment, we consider a colour transfer problem. The goal is to
compute the barycentre of the colour distributions of several (here three)
source images, some of which contain outlier colours, and then use this
barycentre as a target measure to modify the colours of a new image (referred to
as the \textit{input} here). \cref{fig:bar_W1_W2_matching} shows the
source images, the \textit{input} image, and the same \textit{input} image after
transferring its colour distribution to that of the colour barycentre of the
source images. The barycentre is computed either for a $\W_1$ cost or for a
$\W_2^2$ cost. This transfer is evaluated on downsampled images, with the RGB
matching of a colour $c$ in the high-resolution image subsequently chosen as
$c+\tau$, where $\tau$ is the colour translation obtained for the closest colour
to $c$ in the downsampled image (this amounts to viewing the matching as a
piecewise constant translation field). \cref{fig:bar_W1_W2_matching_rgb}
shows the colour distributions of the images in the RGB space. We observe that
the $\W_1$ cost enjoys greater robustness to the colour outliers compared to the
usual $\W_2^2$ cost.

\blue{
Mathematically, we view each high-resolution image $I_k$ as a discrete measure
$\nu_k = \frac{1}{N}\sum_{j=1}^{N}\delta_{y_{k,j}} \in \mathcal{P}(\R^3)$ which
represents the point cloud of its pixel colours $(y_{k,j})_j$. We then
downsample into $\hat\nu_k = \frac{1}{n}\sum_{j=1}^{n}\delta_{\hat y_{k,j}}$
(typically by subsampling). A $\W_p$ ($p\in\{1,2\}$) barycentre $\hat\mu =
\frac{1}{n}\sum_{i=1}^n \delta_{\hat x_i}$ is then computed between the measures
$(\hat\nu_k)$ using \cref{alg:fp_H}. We now apply the colour distribution of
this barycentre $\hat\mu$ to a new input image $J$ represented as $\rho =
\frac{1}{N}\sum_{i=1}^N \delta_{z_i}$ and downsampled to $\hat\rho =
\frac{1}{n}\sum_{i=1}^n \delta_{\hat z_i}$. To do this, we compute an OT
permutation $\sigma$ from $\hat\rho$ to $\hat\mu$ and output the image
$\tilde{J} = (\tilde{z}_i)_{i=1}^N$ defined by: 
\[
\forall i \in \llbracket 1, N \rrbracket,\; 
\tilde{z}_i := z_i + (\hat x_{\sigma(\hat i)} - \hat z_{\hat i}),\quad
\hat i := \underset{j \in \llbracket 1, n \rrbracket}{\argmin}\ \|z_i - \hat z_j\|_2^2.
\]
}

\begin{figure}[ht]
    \centering
    \includegraphics[width=.75\linewidth]{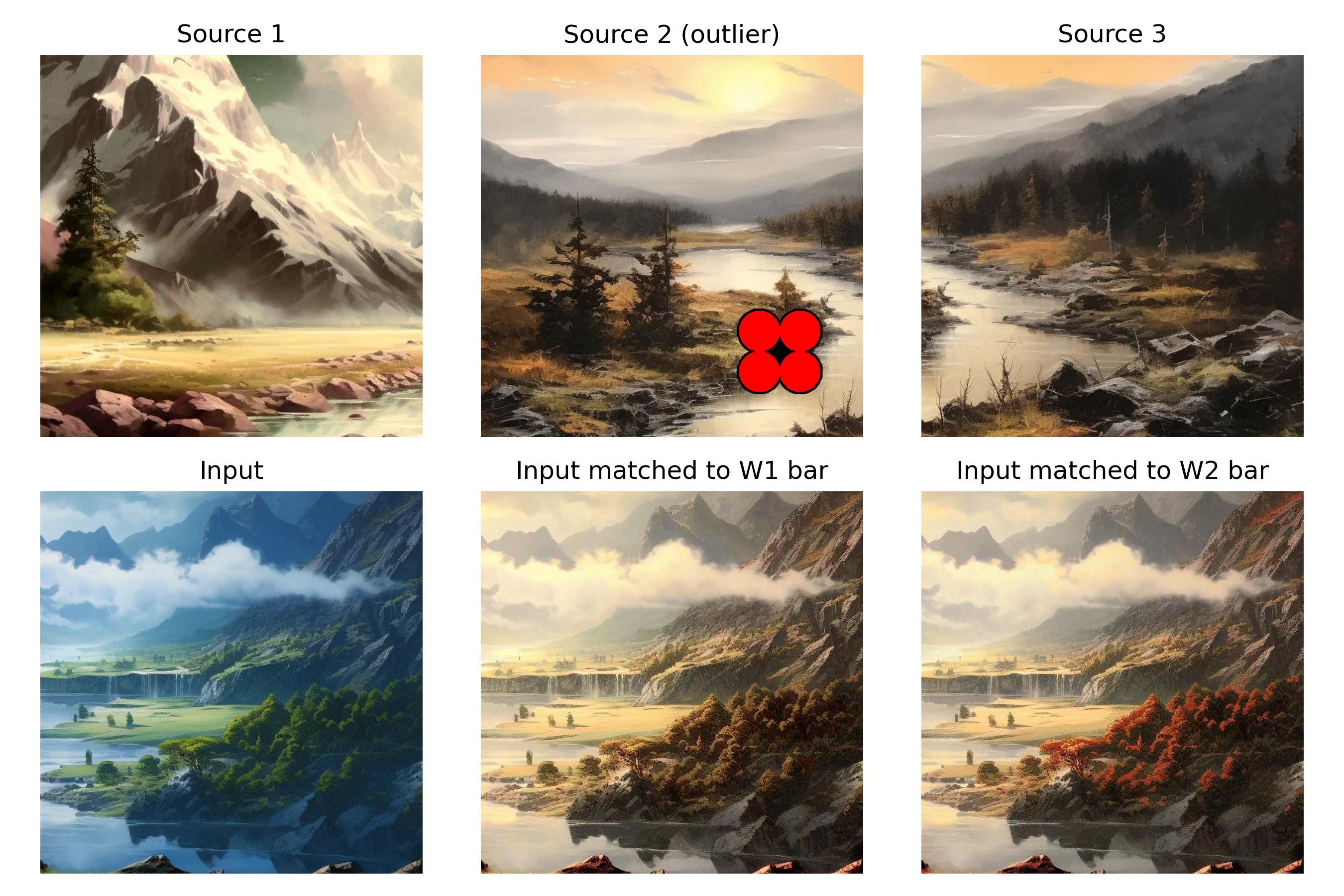}
    \caption{Colour transfer applied to the input image towards the colour
    barycentre of the source images, for the costs $\W_1$ and $\W_2^2$. One of
    the source images contains unwanted colour artifacts, which we see as
    outliers.}
    \label{fig:bar_W1_W2_matching}
\end{figure}

\begin{figure}[ht]
    \centering
    \includegraphics[trim= 0 5em 0 0, clip, width=.75\linewidth]{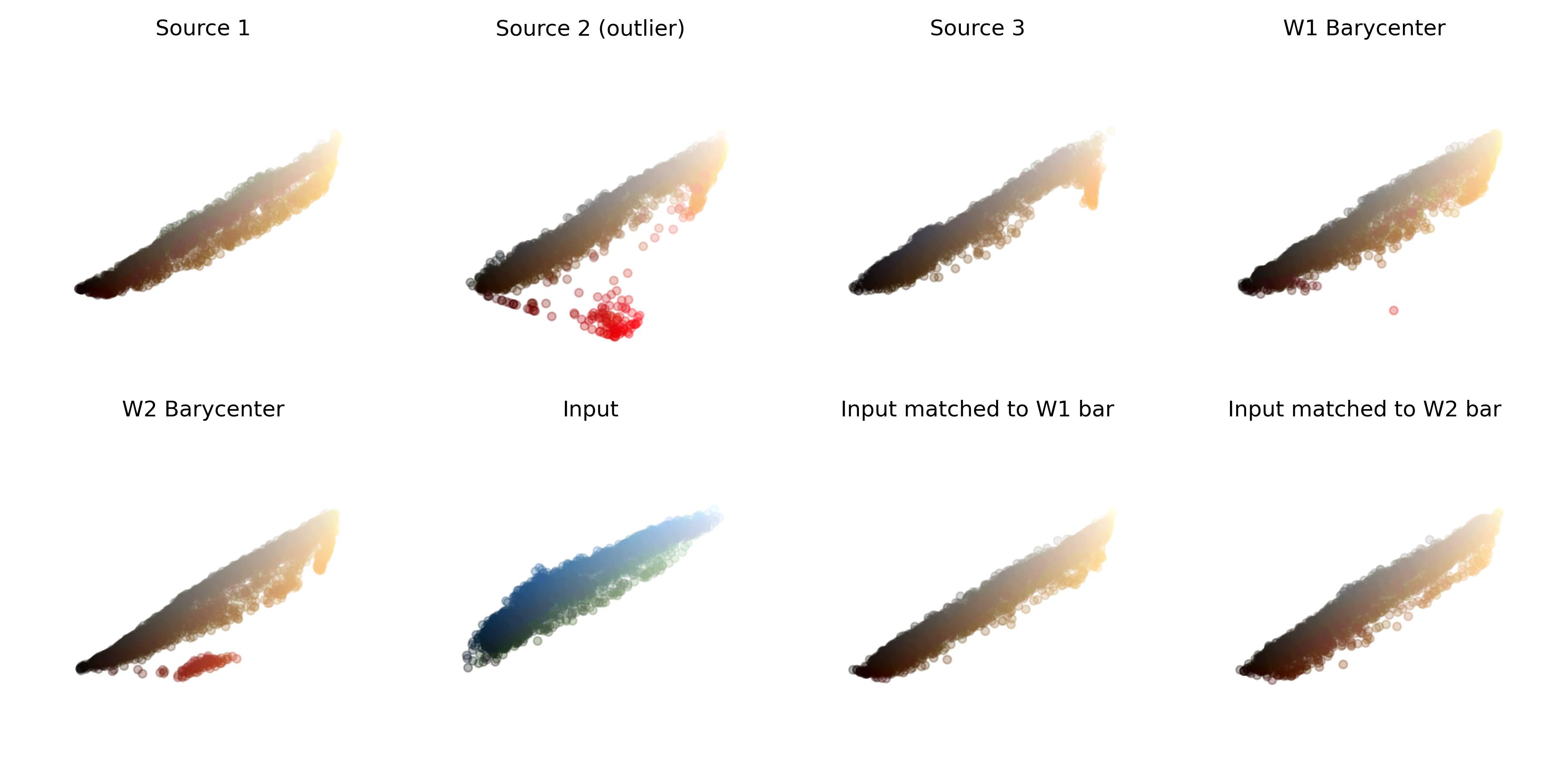}
    \caption{Colour distributions of the different images from
    \cref{fig:bar_W1_W2_matching} as well as the $\W_2^2$ and $\W_1$
    barycentres of the source images.}
    \label{fig:bar_W1_W2_matching_rgb}
\end{figure}

\FloatBarrier
\newpage
\section*{Future Directions}

There are numerous directions for future research. To begin with, in
\cref{thm:fixed_point_iterates_cv}, we show subsequential convergence to
fixed-points of $G$, which may not be barycentres. In cases where barycentres
and fixed points may not be unique such as the discrete setting, it remains
unclear if there exists fixed points that are not barycentres.

The barycentric fixed-point algorithm (iterating \cref{eqn:def_H}) has no
theoretical guarantees of convergence. Given its computational advantages and
its current use in practice for the squared Euclidean cost
(\cite{cuturi14fast,flamary2021pot}), this is a timely question.

In \cref{sec:bar_proj}, we required a notion of barycentric projection for
couplings $\pi \in \Pi_{c_k}^*(\mu, \nu_k)$. In $\R^d$, the underlying convex
combinations are performed using the usual linear structure, however this does
not generalise to arbitrary metric spaces. To consider these object more
formally on generic (compact) metric spaces, it would be necessary to discuss in
more detail the meaning of expectation in a space without a linear structure. 

Throughout this work, we relied heavily on \cref{ass:B}, but in practice this
can be difficult to verify for costs $c_k$: beyond the case $c_k = h(x-y)$ with
$h$ strictly convex, it is difficult to provide large classes of costs that
yield this property on $B$ (other examples include $c_k(x,y) = \|P_k x -
y\|_2^2$ as in \cite{delon2021generalized} or $\W_2^2$ for absolutely continuous
measures). One could alternatively investigate a theoretical framework where $B$
is a multi-function.

In the absolutely continuous case, the Twist condition can ensure uniqueness of
the barycentre, as explained in \cref{rem:uniqueness_matching_teams}. A natural
question concerns almost-sure uniqueness in the discrete case, as was partially
explored in \cref{sec:uniqueness_discrete}.

From a numerical standpoint, it has been observed that the fixed-point algorithm
converges in very few iterations. A theoretical work extending the discrete
Wasserstein case from \cite{lindheim2023simple} would bridge a significant gap
between theory and practical observation.

\subsection*{Acknowledgements}

We would like to thank Christophe Gaillac for the initial discussions that
motivated the introduction of barycentres with generic costs. We thank Nicolas
Juillet for useful discussions about the fixed points of the functionals $G$ and
$H$. This research was funded in part by the Agence nationale de la recherche
(ANR), Grant ANR-23-CE40-0017 and by the France 2030 program, with the reference
ANR-23-PEIA-0004. 

{\small
\printbibliography
}

\end{document}